\documentclass[runningheads]{llncs}
\usepackage[latin1]{inputenc}

\usepackage{amsmath,amsfonts,amssymb}
\usepackage{amsmath}
\usepackage{color,url}
\usepackage{epsfig} 
\usepackage{mathrsfs} 
\usepackage{enumitem}
\usepackage{graphicx}
\usepackage{lineno}
\usepackage{tikz}
\usepackage{tkz-graph}
\begin{document}

\flushbottom\title{The Toll Walk Transit Function of a Graph: Axiomatic Characterizations and First-Order Non-definability}

\author{Manoj Changat$^1$\and Jeny Jacob$^1$ \and Lekshmi Kamal K. Sheela$^1$ \and  Iztok Peterin$^{2,3}$ }
%date{\today \\$ $\\
\institute {\small	$^1$ Department of Futures Studies, University of Kerala, Thiruvananthapuram - 695581, India. E-mails:  mchangat@keralauniversity.ac.in, jenyjacobktr@gmail.com, lekshmisanthoshgr@gmail.com\\
\small	$^2$ University of Maribor, Faculty of Electrical Engineering and Computer Science, Koro\v{s}ka 46, 2000 Maribor, Slovenia.\\ 
\small	$^3$ Institute of Mathematics, Physics and Machanics, Jadranska 19, 1000 Ljubljana, Slovenia\\
E-mail: iztok.peterin@um.si	}
%\begin{document}
\maketitle	
\begin{abstract}
A walk $W=w_1w_2\dots w_k$, $k\geq 2$, is called a toll walk if $w_1\neq w_k$ and $w_2$ and $w_{k-1}$ are the only neighbors of $w_1$ and $w_k$, respectively, on $W$ in a graph $G$. A toll walk interval $T(u,v)$, $u,v\in V(G)$, contains all the vertices that belong to a toll walk between $u$ and $v$. The toll walk intervals yield a toll walk transit function $T:V(G)\times V(G)\rightarrow 2^{V(G)}$. We represent several axioms that characterize the toll walk transit function among chordal graphs, trees,  asteroidal triple-free graphs, Ptolemaic graphs, and distance hereditary graphs. We also show that the toll walk transit function can not be described in the language of first-order logic for an arbitrary graph.
\end{abstract}

\section{Introduction}

A toll walk denoted as $W$ is a type of walk on a graph $G$ that starts at a vertex $u$ and ends at a distinct vertex $v$. It possesses two distinct properties: first, it includes exactly one neighbor of $u$ as its second vertex, and second, it involves exactly one neighbor of $v$ as its penultimate vertex. A toll walk can be likened to a journey with an entrance fee or a toll that is paid only once, specifically at the outset when entering a system represented by a graph. Similarly, one exits the system precisely once, and this occurs at the neighbor of the final vertex.

The concept of toll walks was introduced by Alcon \cite{toll2} as a tool to characterize dominating pairs in interval graphs. Subsequently, Alcon et al. \cite{toll3}, despite the publication year discrepancy, recognized that all vertices belonging to toll walks between $u$ and $v$ could be viewed as the toll interval $T(u,v)$. This led to the development of the toll walk transit function $T: V(G)\times V(G)\rightarrow 2^{V(G)}$ for a graph $G$ and the concept of toll convexity. A pivotal result established in \cite{toll3} asserts that a graph $G$ conforms to the principles of toll convexity if and only if it is an interval graph. Furthermore, research extended to explore toll convexity within standard graph products, examining classical convexity-related invariants, as investigated by Gologranc and Repolusk \cite{GoRe,GoRe1}. More recently, Dourado \cite{Dour} explored the hull number with respect to the toll convexity.

In \cite{lcp} an axiomatic examination of the toll walk function $T$ in a graph was explored. The main tool for this axiomatic approach is the notion of transit function. Mulder \cite{muld-08} introduced transit functions in discrete structures to present a unifying approach for results and ideas on intervals, convexities, and betweenness in graphs, posets, vector spaces, and several other mathematical structures. A transit function is an abstract notion of an interval, and hence the axioms on a transit function are sometimes known as betweenness axioms.  

Specifically, in \cite{lcp} an examination of the various well-established axioms of betweenness along with certain axioms studied in the context of the induced path function, a well-studied transit function on graphs, supplemented by new axioms tailored to the toll walk transit function, was attempted. In addition, in \cite{lcp} a novel axiomatic characterization of interval graphs and 
 subclass of asteroidal triple-free graphs was established. Two problems were posed in \cite{lcp}, which are the following.

\begin{problem}\label{pbm1}
 Is there an axiomatic characterization of the toll walk transit function of an arbitrary connected graph $G$?  
\end{problem}	

\begin{problem}\label{pbm2}
 Is there a characterization of the toll walk transit function of chordal graphs?    
\end{problem}	

In this paper, we solve the Problem~\ref{pbm2} affirmatively and provide the axiomatic characterization of chordal graphs and trees (Section 3), along with
AT-free graphs (Section 4), Ptolemaic graphs (Section 5) and distance-hereditary graphs (Section 5) using the betweenness axioms on an arbitrary transit function $R$. Interestingly, we prove that for the Problem~\ref{pbm1}, there is no characterization of the toll walk transit function of an arbitrary connected graph using a set of first-order axioms. In other words, in Section 6 we prove that the toll walk transit function is not first-order axiomatizable. We use the standard technique of Ehrenfeucht-Fraisse Game of first-order logic to prove the non-definability of the toll walk transit function. In the following section, we settle the notation and recall some known results.

%%%%%%%%%%%%%%%%%%%%%%

\section{Preliminaries}

Let $G$ be a finite simple graph with the vertex set $V(G)$ and the edge set $E(G)$. For a positive integer $k$, we use the notation $[k]$ for the set $\{1,2,\dots,k\}$. The set $\{u\in V(G):uv\in E(G)\}$ is the \emph{open neighborhood} $N(v)$ of $v \in V(G)$ and contains all neighbors of $v$. The \emph{closed neighborhood} $N[v]$ is then $N(v)\cup\{v\}$. A vertex $v$ with $N[V]=V(G)$ is called \emph{universal}. Vertices $w_1,\dots,w_k$ form a \emph{walk} $W_k$ of length $k-1$ in $G$ if $w_iw_{i+1}\in E(G)$ for every $i\in [k-1]$. We simply write $W_k=w_1\cdots w_k$. A walk $W_k$ is called a \emph{path} of $G$ if all vertices of $W_k$ are different. We use the notation $v_1,v_k$-path for a path $P_k=v_1\cdots v_k$ where $P_k$ starts at $v_1$ and ends at $v_k$. Furthermore, $u\xrightarrow{P} x$ denotes the sub-path of a path $P$ with end vertices $u$ and $x$. An edge $v_iv_j$ with $|i-j|>1$ is called a \emph{chord} of $P_k$. A path without chords is an \emph{induced path}. 
The minimum number of edges on a $u,v$-path is the distance $d(u,v)$ between $u,v\in V(G)$. If there is no $u,v$-path in $G$, then we set $d(u,v)=\infty$. A $u,v$-path  of length $d(u,v)$ is called a $u,v$-\emph{shortest path}. 

A walk $W=w_1\cdots w_k$ is called a \emph{toll walk} if $w_1\neq w_k$, $w_2$ is the only neighbor of $w_1$ on $W$ in $G$ and $w_{k-1}$ is the only neighbor of $w_k$ on $W$ in $G$. The only toll walk that starts and ends at the same vertex $v$ is $v$ it itself. The following lemma from \cite{toll3} will be useful on several occasions.

\begin{lemma}\label{toll1} 
A vertex $v$ is in some toll walk between two different non-adjacent vertices $x$ and $y$ if and only if $N[x]-\{v\}$ does not separate $v$ from $y$ and $N[y]-\{v\}$ does not separate $v$ from $x$.
\end{lemma} 

We use the standard notation $C_n$ for a \emph{cycle} on $n\geq 3$ vertices and $K_n$ for a \emph{complete graphs} on $n\geq 1$ vertices. Further graph families that are important to us for $n\geq1$ are \emph{fans} $F_2^{n+1}$ that contain a universal vertex $y_2$ and a path $p_1p_2\dots p_n$, graphs $F_3^{n}$ built by two universal vertices $y_1,y_2$ and a path $p_1p_2\dots p_n$ and $F_4^{n}$ that is obtained from $F_3^{n}$ by deleting the edge $y_1y_2$. In addition, we define the families $XF_2^{n+1}$, $XF_3^{n}$ and $XF_4^{n}$ as follows. We get graph $XF_2^{n+1}$ from $F_2^{n+1}$ by adding vertices $u,v,x$ and edges $up_1,p_nv,y_2x$, similarly we get $XF_3^{n}$ from $F_3^{n}$ by adding vertices $u,v,x$ and edges $up_1,uy_1,vp_n,vy_2,xy_1,xy_2$ and finally we get $XF_4^{n}$ from $F_4^{n}$ by adding vertices $u,v,x$ and edges $up_1,uy_1,vp_n,vy_2,xy_1,xy_2$. Observe $XF_2^{n+1}$, $XF_3^{n}$ and $XF_4^{n}$ in the last three right spots, respectively, in the last line of Figure \ref{fig2}. 

In this work, we often consider classes of graphs that can be described by forbidden induced subgraphs. A graph $G$ is \emph{chordal} if there is no induced cycle of length at least four in $G$ and all chordal graphs form a class of \emph{chordal graphs}. We call cycles of length at least five \emph{holes}. 

Another class of graphs important for us are \emph{distance-hereditary} graphs which are formed by all graphs $G$ in which every induced path in $G$ is also a shortest path in $G$. They also have a forbidden induced subgraphs characterization presented by graphs on Figure \ref{hcdf}, see also Theorem \ref{disther}.

\begin{figure}[ht!]
\begin{center}
\begin{tikzpicture}[scale=0.5,style=thick,x=1cm,y=0.7cm]
\def\vr{3pt} % \vr = vertex radius;

% define vertices
%%%%% %%%%% house
\path (0,0) coordinate (a);
\path (3,0) coordinate (b);
\path (3,3) coordinate (c);
\path (0,3) coordinate (d);
\path (1.5,5) coordinate (e);

%  edges
\draw (a) -- (b) -- (c) -- (d) -- (a);
\draw (c) -- (e) -- (d);

\draw (a) [fill=white] circle (\vr);
\draw (b) [fill=white] circle (\vr);
\draw (c) [fill=white] circle (\vr);
\draw (d) [fill=white] circle (\vr);
\draw (e) [fill=white] circle (\vr);

\draw[anchor = north] (a) node {$x$};
\draw[anchor = north] (b) node {$y$};
\draw[anchor = south] (e) node {$u=v$};
%\draw[anchor = west] (c) node {$w$};
%\draw[anchor = east] (d) node {$z$};
%%%%%%%%%%%%%%%%%%%%C_5

\path (4.5,0) coordinate (a1);
\path (7.5,0) coordinate (b1);
\path (7.5,3) coordinate (c1);
\path (4.5,3) coordinate (d1);
\path (6,5) coordinate (e1);

%  edges
\draw (a1) -- (b1) -- (c1) -- (e1) -- (d1) -- (a1);

\draw (a1) [fill=white] circle (\vr);
\draw (b1) [fill=white] circle (\vr);
\draw (c1) [fill=white] circle (\vr);
\draw (d1) [fill=white] circle (\vr);
\draw (e1) [fill=white] circle (\vr);

\draw[anchor = north] (a1) node {$x$};
\draw[anchor = north] (b1) node {$y$};
\draw[anchor = south] (e1) node {$u=v$};
%\draw[anchor = west] (c1) node {$w$};
%\draw[anchor = east] (d1) node {$z$};

%%%%%%%%%%%%%%%%%%%%hole

\path (9,1.8) coordinate (a1);
\path (12,1.8) coordinate (b1);
\path (12,4.2) coordinate (c1);
\path (9,4.2) coordinate (d1);
\path (10.5,6) coordinate (e1);
\path (10.5,0) coordinate (f1);
%  edges
\draw (a1) -- (f1)-- (b1) -- (c1) -- (e1) -- (d1);
\draw[dotted] (d1) -- (a1);

\draw (a1) [fill=white] circle (\vr);
\draw (b1) [fill=white] circle (\vr);
\draw (c1) [fill=white] circle (\vr);
\draw (d1) [fill=white] circle (\vr);
\draw (e1) [fill=white] circle (\vr);
\draw (f1) [fill=white] circle (\vr);

\draw[anchor = north] (a1) node {$x$};
\draw[anchor = north] (f1) node {$y$};
\draw[anchor = south] (e1) node {$u$};
\draw[anchor = west] (c1) node {$v$};
%\draw[anchor = east] (d1) node {$v$};
%%%%%%%%%%%%%%%%%%%%%%Domino

\path (13.5,0) coordinate (a2);
\path (16.5,0) coordinate (b2);
\path (16.5,3) coordinate (c2);
\path (13.5,3) coordinate (d2);
\path (13.5,6) coordinate (e2);
\path (16.5,6) coordinate (f2);

%  edges
\draw (a2) -- (b2) -- (c2) -- (d2) -- (a2);
\draw (c2) -- (f2) -- (e2) -- (d2);

\draw (a2) [fill=white] circle (\vr);
\draw (b2) [fill=white] circle (\vr);
\draw (c2) [fill=white] circle (\vr);
\draw (d2) [fill=white] circle (\vr);
\draw (e2) [fill=white] circle (\vr);
\draw (f2) [fill=white] circle (\vr);

\draw[anchor = north] (a2) node {$x$};
\draw[anchor = south] (e2) node {$u$};
\draw[anchor = north] (b2) node {$y$};
\draw[anchor = south] (f2) node {$v$};
%%%%%%%%%%%%%%%%5 P

\path (20.5,0) coordinate (a3);
\path (18,4) coordinate (c3);
\path (19.5,5) coordinate (d3);
\path (21.5,5) coordinate (e3);
\path (23,4) coordinate (f3);

%  edges
\draw (a3) -- (c3) -- (d3) -- (e3) -- (f3) -- (a3);
\draw (a3) -- (d3);
\draw (a3) -- (e3);

\draw (a3) [fill=white] circle (\vr);
\draw (c3) [fill=white] circle (\vr);
\draw (d3) [fill=white] circle (\vr);
\draw (e3) [fill=white] circle (\vr);
\draw (f3) [fill=white] circle (\vr);

\draw[anchor = north] (a3) node {$z$};
\draw[anchor = south] (f3) node {$v$};
\draw[anchor = south] (c3) node {$u$};
\draw[anchor = south] (d3) node {$x$};
\draw[anchor = south] (e3) node {$y$};

\end{tikzpicture}
\end{center}
\caption{Graphs house $H$, $C_5$, hole (different form $C_5$), domino $D$ and $3$-fan $F_2^5$ (from left to right).} \label{hcdf}
\end{figure}
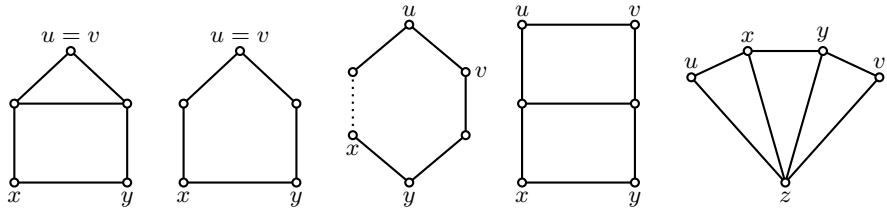 

\begin{theorem}\cite{BaMu}\label{disther} 
A graph $G$ is a distance-hereditary graph if and only if $G$ is $H$hole$DF_2^5$-free.
\end{theorem}

In Section 4 we further define the class of Ptolemaic graphs. Next, we define AT-\emph{free} graphs that contain all asteroidal-triple free graphs. The vertices $u,v,w$ form an \emph{asteroidal triple} in $G$ if there exists a $u,v$ path without a neighbor of $w$, a $u,w$ path without a neighbor of $v$, and a $v,w$ path without a neighbor of $u$. A graph $G$ is called an \emph{AT-free graph} if $G$ does not have an asteroidal triple. The following characterization of $AT$-free graphs with forbidden induced subgraphs from \cite{Kohl}, see also \cite{ATfree}, will be important later. All forbidden induced subgraphs are depicted in Figure \ref{fig2}. We use the same notation as presented in \cite{ATfree}.

\begin{theorem}\cite{Kohl}\label{AT-free} 
A graph $G$ is $(C_kT_2X_2X_3X_{30}\dots X_{41}XF_2^{n+1}XF_3^{n}XF_4^{n})$-free for $k\geq 6$ and $n\geq 1$ if and only if $G$ is $AT$-free graph.
\end{theorem}

\begin{figure}[ht!]
\begin{center}
\begin{tikzpicture}[scale=0.5,style=thick,x=0.8cm,y=0.8cm]
\def\vr{3pt} % \vr = vertex radius;

% define vertices
%%%%% %%%%% C_k
\path (0,0) coordinate (a);
\path (-1.5,2.5) coordinate (b);
\path (0,5) coordinate (c);
\path (2.5,5) coordinate (d);
\path (4,2.5) coordinate (e);
\path (2.5,0) coordinate (f);

%  edges
\draw (a) -- (b) -- (c) -- (d) -- (e) -- (f);

\draw (a) [fill=white] circle (\vr);
\draw (b) [fill=white] circle (\vr);
\draw (c) [fill=white] circle (\vr);
\draw (d) [fill=white] circle (\vr);
\draw (e) [fill=white] circle (\vr);
\draw (f) [fill=white] circle (\vr);

\draw[anchor = east] (a) node {$x$};
\draw[anchor = east] (c) node {$u$};
\draw[anchor = west] (e) node {$v$};

\draw (1,-0.8) node {$C_k$};
\draw (1.25,0) node {$\cdots$};

%%%%%%%%%%%%%%%%%%%% T_2

\path (6,-0.5) coordinate (a1);
\path (7.5,1) coordinate (b1);
\path (9,2.5) coordinate (c1);
\path (9,4) coordinate (d1);
\path (9,5.5) coordinate (e1);
\path (10.5,1) coordinate (f1);
\path (12,-0.5) coordinate (g1);

%  edges
\draw (a1) -- (b1) -- (c1) -- (d1) -- (e1);
\draw (c1) -- (f1) -- (g1);

\draw (a1) [fill=white] circle (\vr);
\draw (b1) [fill=white] circle (\vr);
\draw (c1) [fill=white] circle (\vr);
\draw (d1) [fill=white] circle (\vr);
\draw (e1) [fill=white] circle (\vr);
\draw (f1) [fill=white] circle (\vr);
\draw (g1) [fill=white] circle (\vr);

\draw[anchor = south] (a1) node {$u$};
\draw[anchor = south] (e1) node {$x$};
\draw[anchor = south] (g1) node {$v$};
\draw (9,-0.8) node {$T_2$};

%%%%%%%%%%%%%%%%%%%%%% X_2

\path (13,1.5) coordinate (a2);
\path (15,1.5) coordinate (b2);
\path (16.5,0) coordinate (c2);
\path (16.5,3) coordinate (d2);
\path (16.5,4.5) coordinate (e2);
\path (18,1.5) coordinate (f2);
\path (20,1.5) coordinate (g2);

%  edges
\draw (a2) -- (b2) -- (c2) -- (f2) -- (g2);
\draw (b2) -- (d2) -- (f2);
\draw (d2) -- (e2);

\draw (a2) [fill=white] circle (\vr);
\draw (b2) [fill=white] circle (\vr);
\draw (c2) [fill=white] circle (\vr);
\draw (d2) [fill=white] circle (\vr);
\draw (e2) [fill=white] circle (\vr);
\draw (f2) [fill=white] circle (\vr);
\draw (g2) [fill=white] circle (\vr);

\draw[anchor = north] (a2) node {$u$};
\draw[anchor = north] (b2) node {$a$};
\draw[anchor = north] (g2) node {$v$};
\draw[anchor = west] (d2) node {$b$};
\draw[anchor = south] (e2) node {$x$};
\draw[anchor = north] (f2) node {$c$};
\draw[anchor = south] (c2) node {$d$};
\draw (16.5,-0.8) node {$X_2$};

%%%%%%%%%%%%%%%%5 X_3

\path (22,0) coordinate (a3);
\path (26,0) coordinate (c3);
\path (26,2) coordinate (d3);
\path (24,2) coordinate (e3);
\path (22,2) coordinate (f3);
\path (24,0) coordinate (b3);
\path (24,3.5) coordinate (g3);

%  edges
\draw (a3) -- (b3) -- (c3) -- (d3) -- (e3) -- (f3) -- (a3);
\draw (b3) -- (e3) -- (g3);

\draw (a3) [fill=white] circle (\vr);
\draw (c3) [fill=white] circle (\vr);
\draw (d3) [fill=white] circle (\vr);
\draw (e3) [fill=white] circle (\vr);
\draw (f3) [fill=white] circle (\vr);
\draw (b3) [fill=white] circle (\vr);
\draw (g3) [fill=white] circle (\vr);

\draw[anchor = east] (a3) node {$u$};
\draw[anchor = west] (c3) node {$v$};
\draw[anchor = south] (g3) node {$x$};

\draw (24,-0.8) node {$X_3$};

s
\end{tikzpicture}

%%%%%%%%%%%%%%%%%%%%%%%%%%%%%%%%%%%%%%%%%%%%%%%%%%%%%ž
\begin{tikzpicture}[scale=0.5,style=thick,x=0.8cm,y=0.8cm]
\def\vr{3pt} % \vr = vertex radius;

% define vertices
%%%%% %%%%% X_{30}
\path (0,0) coordinate (a);
\path (2,0) coordinate (b);
\path (4,0) coordinate (c);
\path (6,0) coordinate (d);
\path (2,2) coordinate (e);
\path (4,2) coordinate (f);
\path (3,3.5) coordinate (g);

%  edges
\draw (a) -- (b) -- (c) -- (d);
\draw (c) -- (f) -- (e) -- (b);
\draw (e) -- (g) -- (f);

\draw (a) [fill=white] circle (\vr);
\draw (b) [fill=white] circle (\vr);
\draw (c) [fill=white] circle (\vr);
\draw (d) [fill=white] circle (\vr);
\draw (e) [fill=white] circle (\vr);
\draw (f) [fill=white] circle (\vr);
\draw (g) [fill=white] circle (\vr);

\draw[anchor = north] (a) node {$x$};
\draw[anchor = north] (d) node {$v$};
\draw[anchor = south] (g) node {$u$};
\draw (3,-0.8) node {$X_{30}$};

%%%%%%%%%%%%%%%%%%%%C X_{31}

\path (9,0) coordinate (a1);
\path (11,0) coordinate (b1);
\path (13,0) coordinate (c1);
\path (13,2) coordinate (d1);
\path (11,2) coordinate (e1);
\path (9,2) coordinate (f1);
\path (11,3.5) coordinate (g1);

%  edges
\draw (a1) -- (b1) -- (c1) -- (d1) -- (e1) -- (f1) -- (a1);
\draw (d1) -- (b1) -- (f1);
\draw (b1) -- (e1) -- (g1);

\draw (a1) [fill=white] circle (\vr);
\draw (b1) [fill=white] circle (\vr);
\draw (c1) [fill=white] circle (\vr);
\draw (d1) [fill=white] circle (\vr);
\draw (e1) [fill=white] circle (\vr);
\draw (f1) [fill=white] circle (\vr);
\draw (g1) [fill=white] circle (\vr);

\draw[anchor = north] (a1) node {$u$};
\draw[anchor = north] (c1) node {$v$};
\draw[anchor = south] (g1) node {$x$};

\draw (11,-1.2) node {$X_{31}$};

%%%%%%%%%%%%%%%%%%%%%%D X_{32}

\path (16,0) coordinate (a2);
\path (18,0) coordinate (b2);
\path (20,0) coordinate (c2);
\path (20,2) coordinate (d2);
\path (18,2) coordinate (e2);
\path (16,2) coordinate (f2);
\path (18,3.5) coordinate (g2);

%  edges
\draw (a2) -- (b2) -- (c2) -- (d2) -- (e2) -- (f2) -- (a2) -- (f2) -- (e2) -- (b2);
\draw (g2) -- (e2);
\draw (b2) -- (f2);

\draw (a2) [fill=white] circle (\vr);
\draw (b2) [fill=white] circle (\vr);
\draw (c2) [fill=white] circle (\vr);
\draw (d2) [fill=white] circle (\vr);
\draw (e2) [fill=white] circle (\vr);
\draw (f2) [fill=white] circle (\vr);
\draw (g2) [fill=white] circle (\vr);

\draw[anchor = north] (a2) node {$u$};
\draw[anchor = north] (c2) node {$v$};
\draw[anchor = south] (g2) node {$x$};
\draw (18,-0.8) node {$X_{32}$};

%%%%%%%%%%%%%%%%5 X_{33}

\path (22,0) coordinate (a3);
\path (26,0) coordinate (c3);
\path (26,2) coordinate (d3);
\path (24,2) coordinate (e3);
\path (22,2) coordinate (f3);
\path (24,0) coordinate (b3);
\path (23,3.5) coordinate (g3);

%  edges
\draw (a3) -- (b3) -- (c3) -- (d3) -- (e3) -- (f3) -- (a3);
\draw (b3) -- (e3) -- (g3) -- (f3);

\draw (a3) [fill=white] circle (\vr);
\draw (c3) [fill=white] circle (\vr);
\draw (d3) [fill=white] circle (\vr);
\draw (e3) [fill=white] circle (\vr);
\draw (f3) [fill=white] circle (\vr);
\draw (b3) [fill=white] circle (\vr);
\draw (g3) [fill=white] circle (\vr);

\draw[anchor = north] (a3) node {$u$};
\draw[anchor = north] (c3) node {$v$};
\draw[anchor = south] (g3) node {$x$};

\draw (24,-0.8) node {$X_{33}$};

\end{tikzpicture}

%%%%%%%%%%%%%%%%%%%%%%%%%%%%%%%%%%%%%%%%%%%%%%%%%%%%%ž
\begin{tikzpicture}[scale=0.5,style=thick,x=0.7cm,y=0.8cm]
\def\vr{3pt} % \vr = vertex radius;

% define vertices
%%%%% %%%%% X_{34}
\path (0,0) coordinate (a);
\path (2,0) coordinate (b);
\path (2.5,2) coordinate (c);
\path (3.5,3.75) coordinate (d);
\path (1,5.5) coordinate (e);
\path (-1.5,3.75) coordinate (f);
\path (-0.5,2) coordinate (g);

%  edges
\draw (a) -- (b) -- (c) -- (d) -- (e) -- (f) -- (g) -- (a) -- (e) -- (g) -- (c) -- (e);

\draw (a) [fill=white] circle (\vr);
\draw (b) [fill=white] circle (\vr);
\draw (c) [fill=white] circle (\vr);
\draw (d) [fill=white] circle (\vr);
\draw (e) [fill=white] circle (\vr);
\draw (f) [fill=white] circle (\vr);
\draw (g) [fill=white] circle (\vr);

\draw[anchor = north] (f) node {$u$};
\draw[anchor = north] (b) node {$v$};
\draw[anchor = north] (d) node {$x$};

\draw (1,-0.8) node {$X_{34}$};

%%%%%%%%%%%%%%%%%%%%C X_{35}

\path (4.5,0) coordinate (a1);
\path (8.5,0) coordinate (c1);
\path (8.5,2) coordinate (d1);
\path (6.5,2) coordinate (e1);
\path (4.5,2) coordinate (f1);
\path (6.5,0) coordinate (b1);
\path (5.5,3.5) coordinate (g1);

%  edges
\draw (a1) -- (b1) -- (c1) -- (d1) -- (e1) -- (f1) -- (a1);
\draw (f1) -- (b1) -- (e1) -- (g1) -- (f1);

\draw (a1) [fill=white] circle (\vr);
\draw (c1) [fill=white] circle (\vr);
\draw (d1) [fill=white] circle (\vr);
\draw (e1) [fill=white] circle (\vr);
\draw (f1) [fill=white] circle (\vr);
\draw (b1) [fill=white] circle (\vr);
\draw (g1) [fill=white] circle (\vr);

\draw[anchor = north] (a1) node {$u$};
\draw[anchor = north] (c1) node {$v$};
\draw[anchor = south] (g1) node {$x$};

\draw (6.5,-0.8) node {$X_{35}$};

%%%%%%%%%%%%%%%%%%%%%% X_{36}

\path (9.6,0) coordinate (a2);
\path (11,0) coordinate (b2);
\path (13,0) coordinate (c2);
\path (14.4,0) coordinate (d2);
\path (11,2) coordinate (e2);
\path (13,2) coordinate (f2);
\path (12,3.5) coordinate (g2);

%  edges
\draw (a2) -- (b2) -- (c2) -- (d2) -- (f2) -- (g2) -- (e2) -- (a2);
\draw (b2) -- (f2) -- (e2) -- (c2);

\draw (a2) [fill=white] circle (\vr);
\draw (b2) [fill=white] circle (\vr);
\draw (c2) [fill=white] circle (\vr);
\draw (d2) [fill=white] circle (\vr);
\draw (e2) [fill=white] circle (\vr);
\draw (f2) [fill=white] circle (\vr);
\draw (g2) [fill=white] circle (\vr);

\draw[anchor = north] (a2) node {$u$};
\draw[anchor = north] (d2) node {$v$};
\draw[anchor = south] (g2) node {$x$};

\draw (12,-0.8) node {$X_{36}$};

%%%%%%%%%%%%%%%%5 X_{37}

\path (17.5,0) coordinate (a3);
\path (19,1.5) coordinate (c3);
\path (19,3.5) coordinate (d3);
\path (17.5,5) coordinate (e3);
\path (16,3.5) coordinate (f3);
\path (16,1.5) coordinate (b3);

%  edges
\draw (a3) -- (c3) -- (d3) -- (e3) -- (f3) -- (b3) -- (a3);
\draw (d3) -- (f3);

\draw (a3) [fill=white] circle (\vr);
\draw (c3) [fill=white] circle (\vr);
\draw (d3) [fill=white] circle (\vr);
\draw (e3) [fill=white] circle (\vr);
\draw (f3) [fill=white] circle (\vr);
\draw (b3) [fill=white] circle (\vr);

\draw[anchor = north] (b3) node {$x$};
\draw[anchor = north] (c3) node {$v$};
\draw[anchor = south] (e3) node {$u$};
\draw (17.5,-0.8) node {$X_{37}$};

%%%%%%%%%%%%%%%%%%%%%% X_{38}

\path (20.5,0) coordinate (a2);
\path (22.5,0) coordinate (b2);
\path (22.5,2) coordinate (c2);
\path (22.5,4) coordinate (d2);
\path (21.5,5.5) coordinate (e2);
\path (20.5,4) coordinate (f2);
\path (20.5,2) coordinate (g2);

%  edges
\draw (a2) -- (b2) -- (c2) -- (d2) -- (e2) -- (f2) -- (g2) -- (a2);
\draw (g2) -- (c2);

\draw (a2) [fill=white] circle (\vr);
\draw (b2) [fill=white] circle (\vr);
\draw (c2) [fill=white] circle (\vr);
\draw (d2) [fill=white] circle (\vr);
\draw (e2) [fill=white] circle (\vr);
\draw (f2) [fill=white] circle (\vr);
\draw (g2) [fill=white] circle (\vr);

\draw[anchor = east] (f2) node {$x$};
\draw[anchor = west] (d2) node {$v$};
\draw[anchor = north] (b2) node {$u$};
\draw (21.5,-0.8) node {$X_{38}$};

%%%%%%%%%%%%%%%%%%%%%% X_{39}

\path (23.6,0) coordinate (a2);
\path (25,0) coordinate (b2);
\path (27,0) coordinate (c2);
\path (28.4,0) coordinate (d2);
\path (25,2) coordinate (e2);
\path (27,2) coordinate (f2);
\path (26,3.5) coordinate (g2);

%  edges
\draw (a2) -- (b2) -- (c2) -- (d2) -- (f2) -- (g2) -- (e2) -- (a2);
\draw (b2) -- (f2);
\draw (e2) -- (c2);

\draw (a2) [fill=white] circle (\vr);
\draw (b2) [fill=white] circle (\vr);
\draw (c2) [fill=white] circle (\vr);
\draw (d2) [fill=white] circle (\vr);
\draw (e2) [fill=white] circle (\vr);
\draw (f2) [fill=white] circle (\vr);
\draw (g2) [fill=white] circle (\vr);

\draw[anchor = north] (a2) node {$u$};
\draw[anchor = north] (d2) node {$v$};
\draw[anchor = south] (g2) node {$x$};
\draw (26,-0.8) node {$X_{39}$};

\end{tikzpicture}

%%%%%%%%%%%%%%%%%%%%%%%%%%%%%%%%%%%%%%%%%%%%%%%%%%%%%ž
\begin{tikzpicture}[scale=0.5,style=thick,x=0.7cm,y=0.8cm]
\def\vr{3pt} % \vr = vertex radius;

% define vertices
%%%%%%%%%%%%%%%%%%%%%% X_{40}

\path (0.6,0) coordinate (a);
\path (2,0) coordinate (b);
\path (4,0) coordinate (c);
\path (5.4,0) coordinate (d);
\path (2,2) coordinate (e);
\path (4,2) coordinate (f);
\path (3,3.5) coordinate (g);

%  edges
\draw (a) -- (b) -- (c) -- (d) -- (f) -- (g) -- (e) -- (a);
\draw (b) -- (f);
\draw (e) -- (c) -- (f);

\draw (a) [fill=white] circle (\vr);
\draw (b) [fill=white] circle (\vr);
\draw (c) [fill=white] circle (\vr);
\draw (d) [fill=white] circle (\vr);
\draw (e) [fill=white] circle (\vr);
\draw (f) [fill=white] circle (\vr);
\draw (g) [fill=white] circle (\vr);

\draw[anchor = north] (a) node {$u$};
\draw[anchor = north] (d) node {$v$};
\draw[anchor = south] (g) node {$x$};
\draw (3,-1.5) node {$X_{40}$};

%%%%%%%%%%%%%%%%%%%%C X_{41}

\path (7,0) coordinate (a1);
\path (8.5,0) coordinate (b1);
\path (10.5,0) coordinate (c1);
\path (12,0) coordinate (d1);
\path (8.5,2) coordinate (e1);
\path (10.5,2) coordinate (f1);
\path (9.5,3.5) coordinate (g1);

%  edges
\draw (a1) -- (b1) -- (c1) -- (d1);
\draw (c1) -- (f1) -- (g1) -- (e1) -- (b1);

\draw (a1) [fill=white] circle (\vr);
\draw (c1) [fill=white] circle (\vr);
\draw (d1) [fill=white] circle (\vr);
\draw (e1) [fill=white] circle (\vr);
\draw (f1) [fill=white] circle (\vr);
\draw (b1) [fill=white] circle (\vr);
\draw (g1) [fill=white] circle (\vr);

\draw[anchor = north] (a1) node {$x$};
\draw[anchor = north] (d1) node {$v$};
\draw[anchor = south] (g1) node {$u$};

\draw (9.5,-1.5) node {$X_{41}$};

%%%%%%%%%%%%%%%%%%%%%% XF_2^{n+1}

\path (13,0) coordinate (a2);
\path (14.5,0) coordinate (b2);
\path (15.5,0) coordinate (c2);
\path (17.5,0) coordinate (d2);
\path (19,0) coordinate (e2);
\path (16,2) coordinate (f2);
\path (16,3.5) coordinate (g2);

%  edges
\draw (a2) -- (b2) -- (c2);
\draw (e2) -- (d2) -- (f2) -- (g2);
\draw (b2) -- (f2) -- (c2);

\draw (a2) [fill=white] circle (\vr);
\draw (b2) [fill=white] circle (\vr);
\draw (c2) [fill=white] circle (\vr);
\draw (d2) [fill=white] circle (\vr);
\draw (e2) [fill=white] circle (\vr);
\draw (f2) [fill=white] circle (\vr);
\draw (g2) [fill=white] circle (\vr);

\draw[anchor = north] (a2) node {$u$};
\draw[anchor = north] (e2) node {$v$};
\draw[anchor = south] (g2) node {$x$};
\draw[anchor = west] (f2) node {$y_2$};
\draw[anchor = north] (b2) node {$p_1$};
\draw[anchor = north] (c2) node {$p_2$};
\draw[anchor = north] (d2) node {$p_n$};

\draw (16,-1.5) node {$XF_2^{n+1}$};
\draw (16.5,0) node {$\cdots$};

%%%%%%%%%%%%%%%%5 XF_3^{n}

\path (20.1,0) coordinate (a3);
\path (21.5,0) coordinate (b3);
\path (23.5,0) coordinate (c3);
\path (24.9,0) coordinate (d3);
\path (21.5,2) coordinate (e3);
\path (23.5,2) coordinate (f3);
\path (22.5,3.5) coordinate (g3);

%  edges
\draw (a3) -- (b3) -- (e3) -- (f3);
\draw (c3) -- (d3) -- (f3) -- (g3) -- (e3) -- (a3);
\draw (b3) -- (f3);
\draw (e3) -- (c3) -- (f3);

\draw (a3) [fill=white] circle (\vr);
\draw (b3) [fill=white] circle (\vr);
\draw (c3) [fill=white] circle (\vr);
\draw (d3) [fill=white] circle (\vr);
\draw (e3) [fill=white] circle (\vr);
\draw (f3) [fill=white] circle (\vr);
\draw (g3) [fill=white] circle (\vr);

\draw[anchor = north] (a3) node {$u$};
\draw[anchor = north] (d3) node {$v$};
\draw[anchor = south] (g3) node {$x$};
\draw[anchor = north] (b3) node {$p_1$};
\draw[anchor = north] (c3) node {$p_n$};
\draw[anchor = east] (e3) node {$y_1$};
\draw[anchor = west] (f3) node {$y_2$};
\draw (22.5,-1.5) node {$XF_3^{n}$};
\draw (22.5,0) node {$\cdots$};

%%%%%%%%%%%%%%%%%%%%%% XF_4^{n}

\path (26.1,0) coordinate (a4);
\path (27.5,0) coordinate (b4);
\path (29.5,0) coordinate (c4);
\path (30.9,0) coordinate (d4);
\path (27.5,2) coordinate (e4);
\path (29.5,2) coordinate (f4);
\path (28.5,3.5) coordinate (g4);

%  edges
\draw (a4) -- (b4);
\draw (c4) -- (d4) -- (f4) -- (g4) -- (e4) -- (a4);
\draw (e4) -- (b4) -- (f4);
\draw (e4) -- (c4) -- (f4);

\draw (a4) [fill=white] circle (\vr);
\draw (b4) [fill=white] circle (\vr);
\draw (c4) [fill=white] circle (\vr);
\draw (d4) [fill=white] circle (\vr);
\draw (e4) [fill=white] circle (\vr);
\draw (f4) [fill=white] circle (\vr);
\draw (g4) [fill=white] circle (\vr);

\draw[anchor = north] (a4) node {$u$};
\draw[anchor = north] (d4) node {$v$};
\draw[anchor = south] (g4) node {$x$};
\draw[anchor = north] (b4) node {$p_1$};
\draw[anchor = north] (c4) node {$p_n$};
\draw[anchor = east] (e4) node {$y_1$};
\draw[anchor = west] (f4) node {$y_2$};

\draw (28.5,-1.5) node {$XF_4^{n}$};
\draw (28.5,0) node {$\cdots$};

\end{tikzpicture}

\end{center}
\caption{Forbidden induced subgraphs of $AT$-free graphs for $k\geq 6$ and $n\geq 1$.} \label{fig2}
\end{figure}
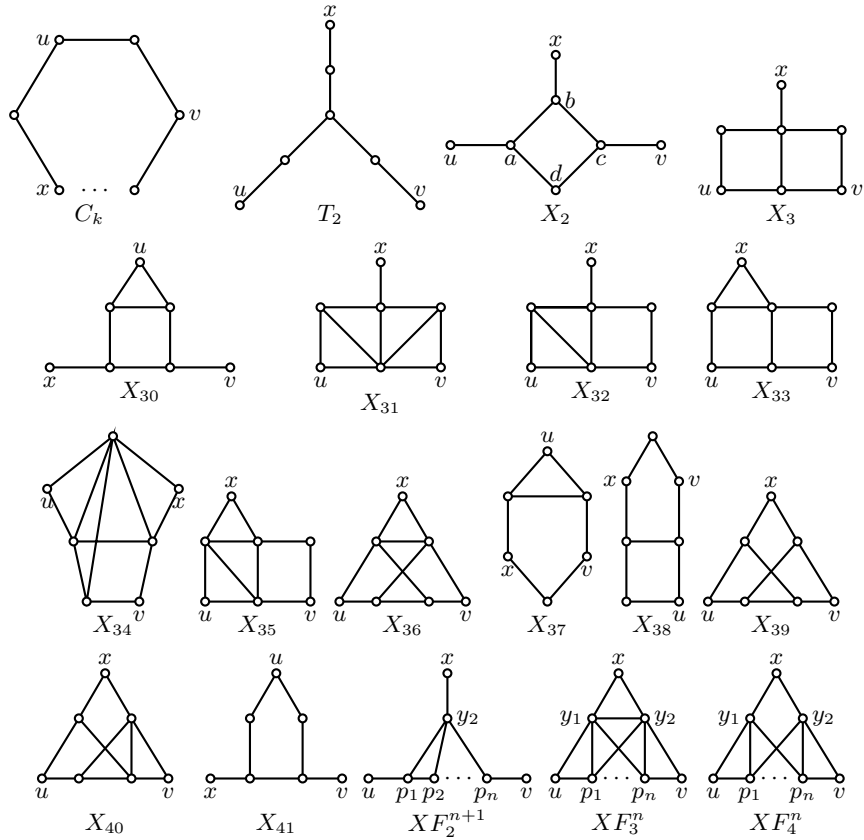

We continue with the formal definition of a transit function.  	 
A \emph{transit function} on a set $V$ is a function $R : V \times V \longrightarrow 2^{V}$ such that for every $u,v \in V$ the following three conditions hold:
	 \begin{itemize}
	 	\item[(t1)]  $u \in R(u,v)$;
	 	\item[(t2)]  $R(u,v)=R(v,u)$;
	 	\item[(t3)]  $R(u,u)=\{u\}$.
	 \end{itemize}
The \emph{underlying graph} $G_{R}$ of a transit function $R$ is a graph with vertex set $V$, where distinct vertices $u$ and $v$ are adjacent if and only if $R(u,v)=\{u,v\}$.

The well studied transit functions in graphs are the interval function $I_G$, induced path function $J_G$ and the all paths function $A_G$.  The \emph{interval function} $I_G$ of a connected graph $G$ is defined with respect to the standard distance $d$ in $G$ as $I: V\times V \longrightarrow 2^{V}$ where
	 	$$I_G(u,v)=\{w\in V(G): w \text{ lies on some }u,v\text{-shortest path in }G \}.$$
The \emph{induced path transit function} $J(u,v)$ of $G$ is a natural generalization of the interval function and is defined as $$J(u, v) =\{w \in V(G) :w\text{ lies on an induced }u,v\text{-path}\}.$$

The well known is also the \emph{ all-path transit function} $A(u, v)=\{w\in V(G):w\text{ lies on a }u,v\text{-path}\}$, see \cite{msh}, which consists of the vertices lying on at least one $u,v$-path. For any two vertices $u$ and $v$ of a connected graph $G$, it is clear that $I(u,v)\subseteq J(u,v)\subseteq A(u,v)$.\\

Probably, the first approach to the axiomatic description of a transit function $I_G$ for a tree $G$ goes back to Sholander \cite{Shol}. His work was later improved by Chv\'atal et al \cite{ChRS}. A full characterization of $I_G$ for a connected graph $G$ was presented by Mulder and Nebesk\'y \cite{mune-09}. They used (t1) and (t2) and  three other betweenness axioms. The idea of the name, "betweenness", is that $x\in R(u,v)$ can be reinterpreted as $x$ is between $u$ and $v$. Two of the axioms of Mulder \cite{muld-08} are important for our approach and follow for a transit function $R$.
\medskip
 
\noindent\textbf{Axiom  (b1).} If there exist elements $u,v,x\in V$ such that $x\in R(u,v), x\neq v$, then $v\notin R(x,u)$.\medskip

\noindent\textbf{Axiom  (b2).} If there exist elements $u,v,x\in V$ such that $x\in R(u,v)$, then $R(u,x)\subseteq R(u,v)$. \medskip

An axiomatic characterization of the induced path transit function $J$ for several classes of graphs, including chordal graphs, was presented in \cite{mcjmhm-10}. These characterizations also use axioms (b1) and (b2) together with other axioms. Some of these axioms are the following.\medskip

\noindent\textbf{Axiom (J0).} If there exist different elements $u,x,y,v \in V$ such that $x\in R(u,y)$ and $y\in R(x,v)$, then $x\in  R(u,v)$.\medskip

\noindent\textbf{Axiom (J2).} If there exist elements $u,v,x\in V$ such that $R(u,x)=\{u,x\}$, $R(x,v)=\{x,v\},u\neq v$ and $R(u,v)\neq\{u,v\}$, then $x\in R(u,v)$.\medskip
 
\noindent\textbf{Axiom (J3).} If there exist elements $u,v,x,y\in V$ such that $x\in R(u,y)$, $y\in R(x,v)$, $x\neq y$ and $R(u,v)\neq \{u,v\}$, then $x\in R(u,v)$.\medskip

\medskip

The following axioms from \cite{lcp} were used to characterize the toll walk transit function of the interval graphs and the AT-free graphs. Here, we correct a small error from \cite{lcp} and add to Axiom (TW1) two additional conditions that $u\neq x$ and $v\neq y$ which are clearly needed. 
\medskip

\noindent\textbf{Axiom (TW1).} If there exist elements $u,v,x,y,z$ such that $x,y\in R(u,v)$, $u\neq x\neq y\neq v$, $R(x,z)=\{x,z\}$, $R(z,y)=\{z,y\}$, $R(x,v)\neq \{x,v\}$ and $R(u,y)\neq \{u,y\}$, then $z\in R(u,v)$.\medskip
		
\noindent\textbf{Axiom (TW2).} If there exist elements $u,v,x,z$ such that $x\in R(u,v)$, $R(u,x)\neq\{u,x\}$, $R(x,v)\neq\{x,v\}$ and $R(x,z)=\{x,z\}$, then $z\in R(u,v)$.\medskip

\noindent\textbf{Axiom (TW3).} If there exist different elements $u,v,x$ such that $x\in R(u,v)$, then there exist $v_1 \in R(x,v), v_1 \neq x$ with $R(x,v_1) = \{x,v_1\}$ and $R(u,v_1) \neq \{u,v_1\}$.  \medskip

Notice that if $R(x,v)=\{x,v\}$, then $v_1=v$ when Axiom (TW3) holds. \medskip

The next axiom is a relaxation of Axiom (b1).\\
\noindent\textbf{Axiom  (b1').} If there exist elements $u,v,x\in V$ such that $x \in R(u,v), v\neq x$ and $R(v,x)\neq\{v,x\}$, then $v\notin R(u,x)$.\medskip

The following corollary is from \cite{lcp}

\begin{corollary}\label{ATDCH2-free}
The toll walk transit function $T$ on a graph $G$ satisfies Axiom (b1') if and only if $G$ is $AT$-free.
\end{corollary}

%%%%%%%%%%%%%%%%%%%%%%%%%%%%%%%%%%%%%%%%%%%%%%%%%%%%%%%%

\section{Toll walk transit function of chordal graphs}

We start with a slight modification of the Axioms (TW3) and (J0) to gain characterization of the toll walk function of chordal graphs.\medskip

\noindent\textbf{Axiom (TWC).} If there exist different elements $u,v,x$ such that $x\in R(u,v)$, then there exist $v_1 \in R(x,v), v_1 \neq x$ with $R(x,v_1) = \{x,v_1\}$, $R(u,v_1) \neq \{u,v_1\}$ and $x\notin R(v_1,v)$.  \medskip

\noindent\textbf{Axiom (JC).} If there exist different elements $u,x,y,v \in V$ such that $x\in R(u,y)$, $y\in R(x,v)$ and $R(x,y)= \{x,y\}$, then $x\in  R(u,v)$.\medskip

 From the definition of Axiom (TWC), it is clear that Axiom (TW3) implies Axiom (TWC), and Axiom (J0) implies Axiom (JC). Furthermore, the Axiom (JC) is symmetric with respect to $x$ and $y$ and at the same time $u$ and $v$ in the sense that we can exchange them. In addition,  it is easy to see that the toll walk transit function does not satisfy the Axiom (TWC) and (JC)  of an arbitrary graph. For instance, Axiom (JC) is not fulfilled on a four-cycle $uxyvu$ and Axiom (TWC) does not hold on a six-cycle $uxv_1avbu$. The next proposition shows that $T$ satisfies the axiom (TWC) on the chordal graphs.
 
\begin{proposition}\label{prop1}
The toll walk transit function satisfies Axiom (TWC) on chordal graphs.
\end{proposition}

\begin{proof}
  Suppose $x\in T(u,v)$. There exists an induced $x,v$-path $P$ that avoids the neighborhood of $u$ with the possible exception of $x$. For the neighbor $v_1$ of $x$ on $P$ it follows that $v_1 \in T(x,v), v_1 \neq x$ with $T(x,v_1) = \{x,v_1\}$ and $T(u,v_1) \neq \{u,v_1\}$. If $v_1=v$, then clearly $x\notin T(v,v_1)=\{v\}$. Similarly, if $T(v_1, v)=\{v_1, v\}$, then $x\notin T(v_1,v)$. Consider next that $T(v_1, v)\neq \{v_1, v\}$. We will show that $x\notin T(v_1,v)$ for a chordal graph $G$. To avoid contradiction, assume that $x\in T(v_1,v)$. There exists an induced $x,v$-path $Q$ that avoids the neighborhood of $v_1$. Let $x_1$ be the neighbor of $x$ on $Q$. Clearly, $x_1v_1 \notin E(G)$. Let $v_2\neq v$ be the neighbor of $v_1$ on $P$ that exists since $T(v_1, v)\neq \{v_1, v\}$. Since $P$ is induced $T(x,v_2)\neq \{x, v_2\}$. If $x_1$ is adjacent to $v_2$, then $xv_1v_2x_1x$ is an induced four-cycle. Otherwise, the path $x_1xv_1v_2$ is part of a larger induced cycle of length greater than four (together with some other vertices of $P$ or $Q$).  Both are not possible in chordal graphs. Hence, $x\notin T(v_1,v)$ and $T$ satisfies Axiom (TWC) on chordal graphs. \hfill\qed\medskip
 \end{proof}
 
\begin{theorem}\label{ch}
The toll walk transit function $T$ satisfies Axiom (JC) on a graph $G$ if and only if $G$ is a chordal graph.
\end{theorem}
\begin{proof}
 Suppose that $G$ contains an induced cycle $C_n$, $n\geq 4$, with consecutive vertices $y,x,u,v$ of $C_n$. Clearly $x\in T(u,y)$, $y\in T(x,v)$, and $T(x,y)=\{x,y\}$ but $x\notin T(u,v)$ since $uv$ is an edge in $G$. That is, if $T$ satisfies Axiom (JC), then $G$ is $C_n$-free for $n\geq 4$.

 Conversely, suppose that $T$  does not satisfy Axiom (JC) on $G$. There exist distinct vertices $u,x,y,v$ such that $x\in T(u,y)$, $y\in T(x,v)$, $T(x,y)= \{x,y\}$ and $x\notin T(u,v)$. Clearly, $x,y,u,v$ belong to the same connected component and there exists an induced $u,x$-path $P$ and an induced $v,y$-path $Q$. Moreover, by $x\in T(u,y)$ we may assume that the only neighbor of $y$ on $P$ is $x$. Similarly, by $y\in T(x,v)$ we may assume that the only neighbor of $x$ on $Q$ is $y$. Now, $x\notin T(u,v)$ implies that $N(u)-x$ separate $x$ from $v$ or $N(v)-x$  separate $u$ from $x$ by the lemma \ref{toll1}. 
 
 By the symmetry of Axiom (JC), we may assume that $N(u)-x$ separates $x$ from $v$. So, every $x,v$-path contains at least one neighbor of $u$. But $x$ belongs to a $u,v$-walk, say $W$, formed by $P$, the edge $xy$ and $Q$. Since $x\notin T(u,v)$, there exists a neighbor of $u$, say $u_1\neq y$, that belongs to $Q$. We may choose $u_1$ to be the first vertex on $Q$ that is adjacent to $u$ after $y$.

 If the cycle $u\xrightarrow{P}xy\xrightarrow{Q} u_1u$ is induced, then $G$ is not chordal, and we are done. Otherwise, there must be some chords from the vertices of $P$ to the vertices of $Q$ different from $y$.  Let $a$ be the last vertex on $P$ before $x$ that is adjacent to some vertex, say $b$, on the $u_1,y$-subpath of $Q$. We may choose $b$ to be closest to $y$ on $Q$ among all such vertices. Since $ab\in E(G)$, we have $b\neq y$ and $a\xrightarrow{P} xy\xrightarrow{Q}ba$ is an induced cycle of length at least four. So, $G$ is not chordal and we are done again.\hfill\qed\medskip
\end{proof}

\begin{lemma}\label{easy1}
Let $R$ be a transit function on a non-empty finite set $V$ satisfying Axioms (J2),  (JC) and  (TW2). If $P_n$, $n\geq 2$, is an induced $u,v$-path in $G_R$, then $V(P_n)\subseteq R(u,v)$. Moreover, if $z$ is adjacent to an inner vertex of $P_n$ that is not adjacent to $u$ or to $v$ in $G_R$, then $z\in R(u,v)$. 
\end{lemma}

\begin{proof}
If $n=2$, then $P_2=uv$ and $R(u,v)=\{u,v\}$ by the definition of $G_R$. If $n=3$, then $P_3=ux_1v$ and $x_1\in R(u,v)$ by Axiom (J2). For $n\geq 4$ we continue by induction. For the basis, let $n=4$ and $P_4=ux_1x_2v$. By Axiom (J2) we have $x_1\in R(u,x_2)$ and $x_2\in R(x_1,v)$. Now, the Axiom (JC) implies that $x_1,x_2\in R(u,v)$. 
Let now $n>4$ and $P_n=ux_1x_2\dots x_{n-1}v$. By the induction hypothesis we have $\{u,x_1,x_2,\dots,x_{n-1}\}\subseteq R(u,x_{n-1})$ and $\{x_1,x_2,\dots,x_{n-1},v\}\subseteq R(x_1,v)$. That is, $x_i\in R(u,x_{i+1})$ and $x_{i+1}\in R(x_i,v)$ for every $i\in [n-2]$. By Axiom (JC) we get $x_i,x_{i+1}\in R(u,v)$ for every $i\in [n-2]$.

For the second part, let $z$ be a neighbor of $x_i$, $i\in\{2,\dots,n-2\}$ that is not adjacent to $u,v$. Clearly, in this case, $n\geq 5$. By the first part of the proof, we have $x_i\in R(u,v)$ and we have $z\in R(u,v)$ by Axiom (TW2).  \qed\medskip
\end{proof}

\begin{proposition}\label{jcj2}
Let $R$ be any transit function defined on a non-empty set $V$. If $R$ satisfies $(JC)$ and $(J2)$, then $G_R$ is chordal.
\end{proposition}

\begin{proof}
Let $R$ be a transit function satisfying $(JC)$ and $(J2)$.  Assume on the contrary that $G_R$ contains an induced cycle, say $C_k=u_1u_2\ldots u_ku_1$, for some $k\geq 4$. Let us first $k=4$. Since $R(u_1,u_2)=\{u_1,u_2\}$ and $R(u_2,u_3)=\{u_2,u_3\}$, we have $u_2\in R(u_1,u_3)$ by Axiom (J2). Similar $u_3\in R(u_2,u_4)$ holds.  Since $R$ satisfies Axiom $(JC)$ we have $u_2\in R(u_1,u_4)$, which is a contradiction as $R(u_1,u_4)=\{u_1,u_4\}$.

Let now $k\geq 5$. Similar to the above, we have $u_{i+1}\in R(u_{i},u_{i+2})$ for every $i\in [n-2]$ and, in particular, $u_{n-1}\in R(u_{n-2},u_n)$. By Lemma \ref{easy1} we have $u_{n-2}\in R(u_{n-1},u_1)$. Now, $u_{n-1}\in R(u_{n-2},u_n)$, $u_{n-2}\in R(u_{n-1},u_1)$ and $R(u_{n-1},u_{n-2})=\{u_{n-1},u_{n-2}\}$ imply that $u_{n-1}\in R(u_1,u_n)$ by Axiom (JC), a contradiction to $R(u_1,u_n)=\{u_1,u_n\}$. \hfill\qed\medskip 
\end{proof}

\begin{theorem} \label{chord}
If $R$ is a transit function on a non-empty finite set $V$ that satisfies Axioms  (b2),  (J2), (JC),  (TW1), (TW2) and (TWC), then $R=T$ on $G_R$. 
\end{theorem}

\begin{proof}
Let $u$ and $v$ be two distinct vertices of $G_R$ and first assume that $x\in R(u,v)$. We have to show that $x\in T(u,v)$ on $G_R$. Clearly $x\in T(u,v)$ whenever $x\in\{u,v\}$. Moreover, if $R(u,v)=\{u,v\}$, then $x$ must be $u$ or $v$. So, assume that $x\notin\{u,v\}$ and that $uv\notin E(G_R)$. If $R(u,x)=\{u,x\}$ and $R(x,v)=\{x,v\}$, then $uxv$ is a toll walk of $G_R$ and $x\in T(u,v)$ follows. Suppose next that $R(x,v)\neq\{x,v\}$. We will construct an $x,v$-path $Q$ in $G_R$ without a neighbor of $u$ (except possibly $x$). For this, let $x=v_0$. By Axiom (TWC) there exists a neighbor of $v_0$, say $v_1$ and $v_1\in R(v_0,v)$ with $R(u,v_1)\neq\{u,v_1\}$ and $v_0 \notin R(v_1,v)$. Since $v_1\in R(v_0,v)$, we have $R(v_1,v)\subseteq R(v_0,v)$ by Axiom (b2) and since $v_0 \notin R(v_1,v)$  we have $R(v_1,v)\subset R(v_0,v)$. In particular, $v_1\in R(v_0,v)\subseteq R(u,v)$ by Axiom (b2). If $v_1\neq v$, then we can continue with the same procedure to get $v_2\in R(v_1, v)$, where $v_2\neq u$, $R(v_1,v_2)=\{v_1,v_2\}$, $R(u,v_2)\neq \{u,v_2\}$ and $v_1 \notin R(v_2,v)$. Furthermore, $R(v_2,v)\subset R(v_1,v)\subset R(v_0,v)$ and $v_2\in R(u,v)$.  Similarly (when $v_2\neq v$), we get $v_3\in R(v_2,v)$ such that $v_3\neq u$, $R(v_2,v_3)=\{v_2,v_3\}$, $R(u,v_3)\neq \{v_1,v_3\}$, $v_2\notin R(v_3,v)$, $v_3\in R(u,v)$ and $R(v_3,v)\subset R(v_2,v)\subset R(v_1,v)\subset R(v_0,v)$. 
Repeating this step, we obtain a sequence of vertices $v_0,v_1,\dots,v_q$, $q\geq 2$, such that
\begin{enumerate}
    \item $R(v_i,v_{i+1})=\{v_i,v_{i+1}\}, i\in\{0,1,\dots,q-1\}$,
    \item $R(u,v_i) \neq \{u,v_i\}, i\in [q]$,
    \item $R(v_{i+1},v) \subset R(v_i,v), i\in\{0,1,\dots,q-1\}$.
\end{enumerate} 
This sequence must stop under the last condition because $V$ is finite. Hence, we may assume that $v_q=v$. Now, if $R(u,x)=\{u,x\}$, then we have a toll $u,v$-walk $uxv_1\dots v_{q-1}v$ and $x\in T(u,v)$. 

If $R(u,x)\neq\{u,x\}$, we can symmetrically build a sequence $u_0,u_1,\dots,u_r$, where $u_0=x$, $u_r=u$, and $u_0u_1\dots u_r$ is a $x,u$-path in $G_R$ that avoids $N[v]$. Clearly, $uu_{r-1}u_{r-2}\dots u_1xv_1\dots v_{q-1}v$ is a toll $u,v$-walk and $x\in T(u,v)$.

Now suppose that $x \in T(u,v)$ and $x\notin\{u,v\}$. We have to show that $x\in R(u,v)$. Let $W$ be a toll $u,v$-walk containing $x$. Clearly, $W$ contains an induced $u,v$-path, say $Q$. If $x$ belongs to $Q$, then $x\in R(u,v)$ by Lemma~\ref{easy1}. So, we may assume that $x$ does not belong to $Q$. The graph $G_R$ is chordal by Proposition~\ref{jcj2}. Let $x_1x_2\dots x_{\ell}$, $x_1\in Q$, $x_{\ell}=x$ be a subpath of $W$ where $x_1$ is the only vertex of $Q$. If $R(u,x_1)=\{u,x_1\}$ and $R(x_1,v)= \{x_1,v\}$, then we have a contradiction with $W$ being a toll $u,v$-walk containing $x$. Without loss of generality, we may assume that $R(x_1,v)\neq \{x_1,v\}$. If also $R(u,x_1)\neq \{u,x_1\}$, then $x\in R(u,v)$ by continuous application of Axiom (TW2) $\ell-1$ times on $x_2,\dots x_{\ell}$. So, let now $R(u,x_1)=\{u,x_1\}$. Since $x$ and $v$ are not separated by $N[u]-\{x\}$ by Lemma \ref{toll1}, there exists an induced $x,v$-path $S$ without a neighbor of $u$. Let $S=s_0s_1\cdots s_k$, $s_0=x$, and $s_k=v$ and let $s_j$ be the first vertex of $S$ that also belongs to $Q$. Let $b$ be the neighbor of $x_1$ on $Q$ different from $u$. By $R(x_1,v)\neq\{x_1,v\}$ we have $b\neq v$. Since $G_R$ is $C_n$ free for every $n\geq4$, the vertex $s_j$ equals $b$ and ($s_{j-1}$ is adjacent to $a$ or $x_2$ is adjacent to $b$). This gives $s_{j-1}\in R(u,v)$ or $x_2\in R(u,v)$, respectively, by Axiom (TW1). Then by continuous application of Axiom (TW2) we have $s_i\in R(u,v)$ for every $i\in\{j-2,j-3,\dots,1,0\}$ or $x_i\in R(u,v)$ for every $i\in\{3,\dots,\ell\}$, respectively. Therefore, $x=s_0=x_{\ell}\in R(u,v)$ and the proof is complete. \qed\medskip
\end{proof}

\begin{proposition}\label{prop2}
Let $T$ be the toll walk transit function on a connected graph $G$. If $T$ satisfies Axiom (JC) on $G$, then $T$ satisfies Axiom (b2).
\end{proposition}

\begin{proof}
Suppose $T$ satisfies Axiom (JC). If $T$ does not satisfy Axiom (b2), then there exists $u,v,x,y$ such that $x\in T(u,v)$, $y\in T(u,x)$ and $y\notin T(u,v)$ must be distinct. Notice that $ux\notin E(G)$ because $y\in T(u,x)$. Since $x\in T(u,v)$, there exists an induced $x,v$-path, say $P$, without a neighbor of $u$ and an induced $x,u$-path, say $Q$, without a neighbor of $v$ (except possibly $x$). Similarly, since $y\in T(u,x)$, there exists an induced $u,y$-path, say $R$, without a neighbor of $x$ (except possibly $y$) and an induced $y,x$-path, say $S$, without a neighbor of $u$ (except possibly $y$). Since $y\notin T(u,v)$, a neighbor of $u$ separates $y$ from $v$ or  
a neighbor of $v$ separates $y$ from $u$ by Lemma \ref{toll1}. But $y\xrightarrow{S} x \xrightarrow{P} v$ is a $y,v$-path that does not contain a neighbor of $u$. Therefore, the only possibility is that a neighbor of $v$ separates $y$ from $u$. Therefore, $R$ contains a neighbor of $v$, say $v'$, which is closest to $y$. The vertices of $v'\xrightarrow{R}y\xrightarrow{S}x\xrightarrow{P}vv'$ contain an induced cycle of length at least four, a contradiction to the Axiom (JC) by Theorem \ref{ch}.\qed \bigskip
\end{proof}

The Axioms (J2), (TW1) and (TW2) are satisfied for a toll walk transit function on any graph $G$. By Theorems~\ref{ch} and \ref{chord}  and Propositions \ref{prop1}, \ref{jcj2} and \ref{prop2} we have the following characterization of the toll walk transit function of a chordal graph.

\begin{theorem}\label{chordal}
A transit function $R$ on a finite set $V$ satisfies the Axioms (b2), (J2), (JC), (TW1), (TW2), and (TWC) if and only if $G_R$ is a chordal graph and $R=T$ on $G_R$.
\end{theorem}
Trees form a special subclass of chordal graphs. To fully describe the toll walk transit function of trees, we define Axiom (tr), which is a generalization of Axiom (J2), and combine it with Axiom (JC).\\
\medskip

 \noindent\textbf{Axiom (tr).} If there exist elements $u,v,x\in V$ such that $R(u,x)=\{u,x\}$, $R(x,v)=\{x,v\}$  then $x\in R(u,v)$.\medskip
\medskip
\begin{lemma}\label{tr}
The toll walk transit function $T$ satisfies Axiom (tr) on a graph $G$ if and only if $G$ is a triangle-free graph.
\end{lemma}
\begin{proof}
   If $G$ contains a triangle with vertices $u, x, v$, then $T(u, x)=\{u, x\}$ and $T(x, v)=\{x, v\}$, but $x\notin T(u, v)$. Therefore, $T$ does not satisfy Axiom (tr).  Conversely, suppose that $T$ does not satisfy the Axiom (tr) on $G$. That is, if $T(u,x)=\{u,x\}$, $T(x,v)=\{x,v\}$, and $x\notin T(u,v)$, then the only possibility is $uv\in E(G)$, which implies that $u,x,v$ forms a triangle.
\end{proof}

By combining Theorem~\ref{ch} and Lemma~\ref{tr}, we obtain the following theorem on the toll walk function of trees.
\begin{theorem}\label{Tr}
The toll walk transit function $T$ satisfies the axioms (JC) and (tr) on a graph $G$ if and only if $G$ is a tree.
\end{theorem}

Now, the characterization of the toll walk transit function of the tree can be obtained by replacing Axiom (J2) with Axiom (tr) in Theorem~\ref{chordal}

\begin{theorem}\label{Tree}
A transit function $R$ on a finite set $V$ satisfies Axioms (b2), (tr), (JC), (TW1), (TW2) and (TWC) if and only if $G_R$ is a tree and $R=T$ on $G_R$.
\end{theorem}
 
%%%%%%%%%%%%%%%%%%%%%%%%%%%%%%%%%%%%%%%%%%%%%%%%%%%

\section{Toll walk transit function of AT-free graphs}

In this Section, we obtain a characterization of the toll function of AT-free graphs. For this we relax Axiom (b2) to (b2') and modify Axiom (J3) to (J4) and (J4'), Axioms (TW1) and (TW2) are generalized by Axiom (TW1') and finally Axiom (TW3) is modified to Axiom (TWA).\medskip

\noindent\textbf{Axiom  (b2').} If there exist elements $u,v,x\in V$ such that $x\in R(u,v)$ and $R(u,x)\neq \{u,x\}$, then $R(x,v)\subseteq R(u,v)$. \medskip

\noindent\textbf{Axiom (J4).} If there exist elements $u,v,x,y\in V$ such that $x\in R(u,y), y\in R(x,v), x\neq y$,  $R(u,x)=\{u,x\}, R(y,v)=\{y,v\}$ and $R(u,v)\neq \{u,v\}$, then $x\in R(u,v)$.\medskip

\noindent\textbf{Axiom (J4').} If there exist elements $u,v,x,y\in V$ such that $x\in R(u,y)$, $y\in R(x,v)$, $R(u,x)\neq\{u,x\}$, $R(y,v)\neq\{y,v\}$, $R(x,y)\neq \{x,y\}$ and $R(u,v)\neq \{u,v\}$, then $x\in R(u,v)$.\medskip

\noindent\textbf{Axiom (TWA).} If there exist different elements $u,v,x$ such that $x\in R(u,v)$, then there exist $x_1 \in R(x,v) \cap R(u,v)$ where $x_1 \neq x$, $R(x,x_1) = \{x,x_1\}$, $R(u,x_1) \neq \{u,x_1\}$ and $R(x_1,v) \subset R(x,v)$.  \medskip

\noindent\textbf{Axiom (TW1').} If there exist elements $u,v,x, w,y,z$ such that $x,y\in R(u,v)$, $x\neq u$, $y\neq v$, $R(x,v)\neq \{x,v\}$, $R(u,y)\neq \{u,y\}$, $R(x,z)=\{x,z\}$, $R(z,w)=\{z,w\}$, $R(w,y)=\{w,y\}$  and $R(u,w)\neq \{u,w\}$, then $z\in R(u,v)$.\medskip

If we have $y=x$ in Axiom (TW1), then $x$ is not adjacent to $u$ nor to $v$. Furthermore, Axiom (TW1) is a special version of Axiom (TW1') if we set $w=z$. When both $w=z$ and $x=y$, we obtain Axiom (TW2) from Axiom (TW1').

\begin{proposition}\label{ba}
The toll walk transit function satisfies Axiom (J4) on any graph $G$. 
\end{proposition}

\begin{proof}
Assume that $x\in T(u,y)$, $y\in T(x,v)$, $x\neq y$,   $T(u,x)=\{u,x\}$, $T(y,v)=\{y,v\}$ and $T(u,v)\neq \{u,v\}$. Since $x\in T(u,y)$, there exists an $x,y$-path $P$ that avoids all neighbors of $u$ beside $x$. Let $a$ be the neighbor of $v$ on $P$ that is closest to $x$ on $P$. (Notice that at least $y$ is a neighbor of $v$ on $P$.) The path $ux\xrightarrow{P}av$ is a $u,v$-toll walk containing $x$ and $x\in T(u,v)$.
\end{proof}

\begin{proposition}\label{j4'}
If $G$ is an AT-free graph and $T$ is the toll walk transit function on $G$, then $T$ satisfies the axioms (J4') and (b2') on $G$.
\end{proposition}

\begin{proof}
First, we show that Axiom (J4') holds. Suppose that $T$ does not satisfy Axiom (J4') on $G$. That is $x\in T(u,y)$, $y\in T(x,v)$, $T(u,x)\neq\{u,x\}$, $T(y,v)\neq\{y,v\}$, $T(x,y)\neq \{x,y\}$, $T(u,v)\neq \{u,v\}$ but $x\notin  T(u,v)$. Since $x\in T(u,y)$ there is a $u,x$-path $P$ without a neighbor of $y$ and an $x,y$-path $Q$ without a neighbor of $u$. Again, since $y\in T(x,v)$, there is a $y,v$-path $R$ without a neighbor of $x$. Since $x\notin  T(u,v)$ we can assume that $N[u]$ separates $x$ from $v$ (the other possibility from Lemma \ref{toll1} is symmetric). That is, $R$ contains a neighbor $u'$ of $u$ and we choose $u'$ to be the neighbor of $u$ on $R$ that is closest to $y$. Then $uu'\xrightarrow{R}y$ is a $u,y$-path without a neighbor of $x$, $P$ is a $u,x$-path without a neighbor of $y$ and $Q$ is an $x,y$-path without a neighbor of $u$. That is, the vertices $u,x,y$ form an asteroidal triple, a contradiction.

Suppose now that $T$ does not satisfy Axiom (b2'). That is, $x\in T(u,v)$, $T(u,x)\neq \{u,x\}$, and $T(x,v)\not\subseteq T(u,v)$. So, there exist $y\in V(G)$ such that $y\in T(x,v)$ and $y\notin T(u,v)$, which means that $y\neq v$, $y\neq x$, and $T(x,v)\neq\{x,v\}$. Since $x\in T(u,v)$, let $P$ be an induced $u,x$-path without a neighbor of $v$ (except possibly $x$) and let $Q$ be an induced $x,v$-path without a neighbor of $u$. Since $y\in T(x,v)$, let $R$ be an induced $x,y$-path without a neighbor of $v$ (except possibly $y$) and $S$ be an induced $y,v$-path without a neighbor of $x$ (except possibly $y$). Furthermore, since $y\notin T(u,v)$, $S$ contains a neighbor $u'$ of $u$. The vertices $u,x,v$ form an asteroidal triple because $P$ is a $u,x$ path without a neighbor of $v$, $Q$ is a $x,v$ path without a neighbor of $u$ and $uu'\xrightarrow{S}v$ is a $u,v$ path without a neighbor of $x$, a contradiction.\qed
\end{proof}

\begin{proposition}\label{Twa}
If $G$ is an AT-free graph, then the toll walk transit function $T$ satisfies Axiom (TWA) on $G$.
\end{proposition}

\begin{proof}
 Suppose that $x\in T(u,v)$ where $u,v,x$ are distinct vertices of an AT-free graph $G$. Let $P$ be an induced $u,x$-path without a neighbor of $v$ (except possibly $x$) and $Q$ be an induced $x,v$-path without a neighbor of $u$ (except possibly $x$). If $T(x,v)=\{x,v\}$, then we are done for $x_1=v$. Otherwise, consider the neighbor $x_1$ of $x$ on $Q$. It follows that $x_1 \in T(x,v),  x_1 \neq x$ with $T(x,x_1) = \{x,x_1\}$ and $T(u,x_1) \neq \{u,x_1\}$. In addition, $u\xrightarrow{P}xx_1\xrightarrow{Q}v$ is a $u,v$-toll walk that contains $x_1$. So, $x_1\in T(u,v)$ and with this $x_1 \in T(u,v) \cap T(x,v)$. We still have to prove that $T(x_1,v) \subset T(x,v)$. If $T(x_1,v)=\{x_1,v\}$, then $x_1$ is the desired vertex and we may assume in what follows that $x_1$ and $v$ are not adjacent. 

 First, we show that $T(x_r,v)\subseteq T(x,v)$ holds for some $x_r$ with $x_r\in T(x,v)\cap T(u,v)$, $T(x_r,x)=\{x_r,x\}$, $x_r\neq x$ and $T(x_r,u)\neq\{x_r,u\}$. If $T(x_1,v)\subseteq T(x,v)$, then we are done. Otherwise, assume that $T(x_1,v) \not\subseteq T(x,v)$ where $y_1\in T(x_1,v) $ and $y_1\notin T(x,v)$. Clearly, $y_1$ is not on $Q$. Let $R_1$ be an induced $x_1,y_1$-path without a neighbor of $v$ (except possibly $y_1$) and $S_1$ be an induced $y_1,v$-path without a neighbor of $x_1$ (except possibly $y_1$). Since $y_1\notin T(x,v)$ $S_1$ contains a neighbor of $x$, say $x_2$, which is closest to $v$ in $S_1$. We claim that $y_1 $ is not adjacent to an internal vertex of the $x_1,v$-subpath of $Q$. If not, then let $y_1'$ be a neighbor of $y_1$ on the $x_1,v$-subpath of $Q$. The walk $xx_1\xrightarrow{R_1}y_1y'_1\xrightarrow{Q}v$ is or contains (when $x$ is adjacent to a vertex of $R_1$ different from $x_1$) a toll $x,v$-walk, a contradiction to $y_1\notin T(x,v)$. Next, we claim that $T(x_1,y_1)= \{x_1,y_1\}$. If not, then $x_1,y_1,v$ form an asteriodal triple since $T(x_1,y_1)\neq \{x_1,y_1\}$ and $R_1$ is an $x_1,y_1$-path without a neighbor of $v$, $x_1\xrightarrow{Q}v$ is an $x_1,v$-path without a neighbor of $y_1$ and $S_1$ is a $y_1,v$-path without a neighbor of $x_1$.  Therefore, $T(x_1,y_1)= \{x_1,y_1\}$ and since $T(u,x_1)\neq \{u,x_1\}$, we have $y_1\neq u$. The next claim is that $u$ is not adjacent to a vertex, say $u_1$, in the $x_2,v$-subpath of $S_1$. If not, then $uu_1\xrightarrow{S_1}v$ is a $u,v$-path without a neighbor of $x_1$, $u\xrightarrow{P}xx_1$ is a $u,x_1$-path without a neighbor of $v$ and $x_1 \xrightarrow{Q}v$ an $x_1,v$-path without a neighbor of $u$. This means that $u$, $v$, and $x_1$ form an asteroidal triple, a contradiction. Therefore, $u$ is not adjacent to a vertex on the $x_2,v$-subpath of $S_1$. In particular, $T(u,x_2)\neq \{u,x_2\}$, $u\xrightarrow{P}xx_2\xrightarrow{S_1}v$ is a toll $u,v$ -walk, and $x_2\in T(u,v)$. If $T(x_2,v)=\{x_2,v\}$, then $x_2$ fulfills Axiom (TWA) and we are done. So, we may assume in what follows that $v$ and $x_2$ are not adjacent. We next claim that $x_2$ is adjacent to some internal vertex, say $x_2'$, of the $x_1, v$-subpath of $Q$.
 If not, then $x_1, x_2,v$ form an asteroidal triple (since $x_2\xrightarrow{S_1}v $ is an $x_2,v$-path without a neighbor of $x_1$, $x_1\xrightarrow{Q}v $ is an $x_1,v$-path without a neighbor of $x_2$ and $x_1xx_2$ is an $x_1,x_2$-path without a neighbor of $v$). Now, $xx_2x_2'\xrightarrow{Q}v$ is a $x,v$-toll walk containing $x_2$. That is $x_2 \in  T(x,v)$ and hence $x_2 \in T(u,v) \cap T(x,v)$. So, if $T(x_2,v) \subseteq T(x,v)$, then $x_2$ is our desired $x_1$. Moreover,  $x_1xx_2\xrightarrow{S_1}v$ is a toll $x_1,v$-walk containing $x_2$. That is $x_2\in T(x_1,v)$ and together with $T(x_1, x_2)\neq \{x_1, x_2\}$, the Axioms (b2') and (b1') yields that $T(x_2,v)\subset T(x_1,v)$. (Recall that Axiom (b1') holds by Corollary \ref{ATDCH2-free} and Axiom (b2') by Proposition \ref{j4'}.)

If not, then there exists $y_2$ (which can be equal to $y_1$) such that $y_2\in T(x_2,v) $ and $y_2\notin T(x,v)$. Since $y_2\in T(x_2,v) $, similar to the above case, let $R_2$ be an induced $x_2,y_2$-path without a neighbor of $v$ (except possibly $y_2$) and $S_2$ be an induced $y_2,v$-path without a neighbor of $x_2$ (except possibly $y_2$). On the other hand,  $y_2\notin T(x,v)$ implies that $S_2$  contains a neighbor of $x$, say $x_3$ (note that $x_3 \neq x_2$ and $T(x_2, x_3)\neq \{x_2, x_3\}$). As in the above case, $T(x_2, y_2)=\{x_2,y_2\}$, otherwise $x_2,y_2,v$ forms an asteriodal triple. In addition, $u$ is not adjacent to a vertex in the $x_3,v$- subpath of $S_2$, otherwise $u,x_2,v$ forms an asteroidal triple. In particular, $T(u,x_3)\neq \{u,x_3\}$, $u\xrightarrow{P}xx_3\xrightarrow{S_2}v$ is a toll $u,v$-walk, and $x_3\in T(u,v)$. If $T(x_3,v)=\{x_3,v\}$, then $x_3$ fulfills Axiom (TWA) and we are done. Therefore, we may assume in what follows that $v$ and $x_3$ are not adjacent. Now we claim that $x_3$ is adjacent to some internal vertices of both $x_2,v$-subpath of $S_1$ and $x_1,v$-subpath of $Q$ otherwise $x_3,x_2,v$ or $x_3,x_1,v$, respectively, form an asteroidal triple. For a neighbor $x_3'$ of $x_3$ in $Q$ is $xx_3x_3'\xrightarrow{Q}v$ a toll $x,v$-walk that contains $x_3$. Hence, $x_3 \in T(u,v) \cap T(x,v)$. So, if $T(x_3,v) \subseteq T(x,v)$, then $x_3$ is our desired $x_1$. If not, then there is $y_3$ (may be $y_1$ or $y_2$) such that $y_3\in T(x_3,v) $ and $y_3\notin T(x,v)$. Since $x_3\in T(x_2,v)$ and $T(x_2, x_3)\neq \{x_2, x_3\}$ we have $T(x_3,v)\subset T(x_2,v)\subset T(x_1,v)$ by Axioms (b2') and (b1'). 

Continuing with this procedure, we get a sequence $x_1,\dots,x_r\in V$ such that $x_r\in T(u,v) \cap T(x,v)$, $T(x,x_r)= \{x,x_r\}$,  $T(u,x_r)\neq \{u,x_r\}$ and $T(x_r,v) \subseteq T(x,v)$ together with $T(x_n,v)\subset \dots \subset T(x_2,v)\subset T(x_1,v)$. This sequence is finite, since $V$ is finite and we may assume that the mentioned sequence is maximal. This means that there does not exist a vertex $w$ in $T(x_n,v)$ such that $w\in T(u,v) \cap T(x,v)$, $T(x,w)\neq \{x,w\}$,  $T(u,w)= \{u,w\}$ and $T(w,v) \subseteq T(x,v)$.

Now we have to prove that $x\notin T(x_r,v)$.  If possible suppose that $x\in T(x_r,v)$, then there exists an induced $x,v$-path, say $P_x$, without a neighbor of $x_r$ (except possibly $x$). Let $v_1$ be the neighbor of $x$ on $P_x$. Now, $x_rxv_1\xrightarrow{P_x}v$  is a toll $x_r,v$-walk containing $v_1$ so that $v_1\in T(x_r,v)$. Also, $T(v_1,x_r)\neq \{v_1, x_r\} $ implies that $T(v_1,v)\subset T(x_r,v)$ by the axioms (b2') and (b1'). Moreover, $T(u,v_1)\neq \{u,v_1\}$, otherwise $u,x_r,v$ form an asteroidal triple.  So, we have $v_1\in T(u,v) \cap T(x,v)$, $T(x,v_1)= \{x,v_1\}$,  $T(u,v_1)\neq \{u,v_1\}$ and $T(v_1,v) \subseteq T(x,v)$, a contradiction to the maximal length of sequence $x_1\dots,x_r$. So $x\notin T(x_r,v)$ and $T(x,v)\subset T(x_r,v)$ and Axiom (TWA) hold for $x_r$.  \qed  
\end{proof}

We continue with a lemma that is similar to Lemma \ref{easy1} only that we use different assumptions now.

\begin{lemma}\label{easy}
Let $R$ be a transit function on a non-empty finite set $V$ satisfying Axioms (J2),  (J4), (J4') and (TW1'). If $P_n$, $n\geq 2$, is an induced $u,v$-path in $G_R$, then $V(P_n)\subseteq R(u,v)$. Moreover, if $z$ is adjacent to an inner vertex of $P_n$ that is not adjacent to $u$ or to $v$ in $G_R$, then $z\in R(u,v)$. 
\end{lemma}

\begin{proof}
If $n=2$, then $P_2=uv$ and $R(u,v)=\{u,v\}$ by the definition of $G_R$. If $n=3$, then $P_3=uxv$ and $x\in R(u,v)$ by Axiom (J2). Let now $n=4$ and $P_4=uxyv$. By Axiom (J2) we have $x\in R(u,y)$ and $y\in R(x,v)$. Now, Axiom (J4) implies that $x,y\in R(u,v)$. If $n=5$, $P_5=uxx_2yv$ and by the previous step, $x,x_2 \in R(u,y)$ and $x_2,y \in R(x,v)$. By Axiom (J4) $x,y\in R(u,v)$ and $x_2\in R(u,v)$ hold by Axiom (TW1') when $z=x_2=w$. If $n=6$, $P_6=uxx_2x_3yv$, then by case $n=5$, we have $\{x,x_2,x_3 \} \in R(u,y)$ and $\{x_2,x_3,y \}\in R(x,v)$. By Axiom (J4) $x,y\in R(u,v)$ and $x_2, x_3 \in R(u,v)$ hold by Axiom (TW1'). For $n=7$, $P_7=uxx_2x_3x_4yv$ by the case $n=5$, $\{x,x_2,x_3 \} \in R(u,x_4)$ and $\{x_3,x_4,y\} \in R(x_2,v)$.  That is $x_2\in R(u,x_4)$, $x_4 \in R(x_2, v)$, $R(u,x_2) \neq \{u,x_2\}$, $R(x_2,x_4) \neq \{x_2,x_4\}$, $R(x_4,v) \neq \{x_4,v\}$, and $R(u,v)\neq \{u,v\}$. By Axiom (J4') we have $x_2, x_4 \in R(u,v)$ and by Axiom (TW1') we have $x,x_3,y \in R(u,v)$.
For a longer path $P_n=uxx_2\dots x_{n-2}yv$, $n>7$, we continue by induction. By the induction hypothesis we have $\{u,x,x_3,\dots,x_{n-2},y\}\subseteq R(u,y)$ and $\{x,x_3,\dots,x_{n-2},y,v\}\subseteq R(x,v)$. In particular, $x_i\in R(u,x_{i+2})$ and $x_{i+2}\in R(x_i,v)$ for every $i\in\{2,\dots,n-4\}$. By Axiom (J4') we get $x_i,x_{i+2}\in R(u,v)$ and by Axiom (TW1') we have $x,x_{i+1},y \in R(u,v)$ for every $i\in\{2,\dots, n-2\}$.

For the second part, let $z$ be a neighbor of $x_i$, $i\in\{2,\dots,n-2\}$ that is not adjacent to $u,v$. Clearly, in this case $n\geq 5$. By the first part of the proof, we have $x_i\in R(u,v)$ and we have $z\in R(u,v)$ by Axiom (TW2) which follows from Axiom (TW1').  \qed\bigskip
\end {proof}

\begin{theorem} \label{at-free}
If $R$ is a transit function on a non-empty finite set $V$ satisfying the Axioms (b1'), (J2),  (J4), (J4') and (TW1'), then $G_R$ is $AT$-free graph.
\end{theorem}

\begin{proof}	
Let $R$ be a transit function satisfying Axioms (b1'), (J2), (J4), (J4') and (TW1'). Axiom (TW1') implies that also Axioms (TW1) and (TW2) hold. We have to prove that $G_R$ is AT-free. By Theorem \ref{AT-free} it is enough to prove that $G_R$ does not contain as an induced subgraph any of the graphs $C_k, T_2,X_2, X_3,$ $X_{30},\dots,X_{41}, XF_2^{n+1}, XF_3^{n},XF_4^{n}$, $k\geq 6$, $n\geq 1$, depicted on Figure~\ref{fig2}. We will show that if $G_R$ contains one of the graphs from Figure \ref{fig2} as an induced subgraph, then we get a contradiction to Axiom (b1'). For this we need to find vertices $u,v,x$ such that $x\in R(u,v)$, $v\neq x$, $R(v,x)\neq\{v,x\}$ and $v\in R(u,x)$. For this, we use vertices $u,v,x$ as marked in Figure \ref{fig2}. Notice that in all graphs of Figure \ref{fig2} we have $v\neq x$ and $R(v,x)\neq\{v,x\}$.

First, we show that $x\in R(v,u)$ holds for all graphs from Figure \ref{fig2}. There exists an induced $u,v$-path that contains $x$ in the graphs $C_k$, $k\geq 6$, $X_{37},X_{38},X_{39},$ $X_{40}$ and $XF_4^n$, $n\geq 1$. By Lemma \ref{easy} $x\in R(u,v)$ for these graphs. For graphs $X_2,X_3,X_{30},X_{31},X_{32},X_{33},$ $X_{35},X_{41}$ and $XF_2^{n+1}$ for $n\geq 2$ there exists an induced $u,v$-path with an inner vertex not adjacent to $u$ nor to $v$, but to $x$. Hence, $x\in R(u,v)$ by Axiom (TW2). Similarly, we see that $x\in R(u,v)$ in $T_2$, only that here we use Axiom (TW2) twice. For graphs $X_{34},X_{36}$ and $XF_3^n$, $n\geq 1$, there exists an induced $u,v$-path such that two different inner vertices are both adjacent to $x$. Thus, $x\in R(u,v)$ by Axiom (TW1). Finally, for $XF_2^2$ we have $y_2\in R(u,v)$ by Axiom (TW1) because it is adjacent to two different inner vertices of an induced $u,v$-path. Now, $x\in R(u,v)$ follows by Axiom (TW2). 

It remains to show that $v\in R(u,x)$ for all graphs in Figure \ref{fig2}. There exists an induced $u,x$-path that contains $v$ in $X_{37},X_{38}$ and $C_k$, $k\geq 6$ and $v\in R(u,x)$ according to the Lemma \ref{easy}. For graphs $X_3,X_{31},X_{32},X_{33},X_{34},X_{35},X_{36},X_{39},X_{40}$ there exists an induced $u,x$-path such that two different inner vertices are adjacent to different adjacent vertices, one of them being $v$. Thus, $v\in R(u,x)$ by Axiom (TW1'). For graphs $X_{30},X_{41}$ exists an induced $u,x$-path with an inner vertex not adjacent to $u$ nor to $x$, but to $v$. Therefore, $v\in R(u,x)$ by Axiom (TW2). Similarly, we see that $v\in R(u,x)$ in $T_2$, only that here we use Axiom (TW2) twice. In $X_2$ we have only one induced $u,x$-path $uabx$. For these four vertices we get $c,d\in R(u,x)$ by Axiom (TW1'). By Axiom (TW2) we get $v\in R(u,x)$. We are left with $XF_2^2,XF_3^n$ and $XF_4^n$, $n\geq 1$. Here $p_1,y_2\in R(u,x)$ since $up_1y_2,x$ is an induced path. Now we use Axiom (TW1) $n-1$ times to get $p_2,\dots,p_n\in R(u,x)$. Finally, $v\in R(u,x)$ by Axiom (TW2) for $XF_2^{n+1}$ and by Axiom (TW1) for $XF_3^n$ and $XF_4^n$. \hfill\qed\bigskip
\end{proof}

\begin{theorem} \label{AT1-free}
If $R$ is a transit function on a non-empty finite set $V$ satisfying the Axioms (b1'), (b2'), (J2), (J4), (J4'), (TW1') and (TWA) then $T= R$ on $G_R$. 
\end{theorem}

\begin{proof}
Let $u$ and $v$ be two distinct vertices of $G_R$ and first assume that $x\in R(u,v)$. We have to show that $x\in T(u,v)$ on $G_R$. Clearly $x\in T(u,v)$ whenever $x\in \{u,v\}$. So, assume that $x\notin \{u,v\}$. If $R(u,x)=\{u,x\}$ and $R(x,v)=\{x,v\}$, then $uv\notin E(G_R)$ by the definition of $G_R$. Therefore, $uxv$ is a toll walk of $G_R$ and $x\in T(u,v)$ follows. Suppose next that $R(x,v)\neq \{x,v\}$. We will construct an $x,v$-path $Q$ in $G_R$ without a neighbor of $u$ (except possibly $x$). 
For this, let $x=x_0$. By Axiom (TWA) there exists a neighbor $x_1$ of $x_0$ where $x_1\in R(x_0,v)\cap R(u,v)$, $R(u,x_1)\neq\{u,x_1\}$ and $T(x_1,v) \subset T(x,v)$.
Since $x_1\neq v$ and $x_1\in R(u,v)$, we can continue with the same procedure to get $x_2 \in R(u,v) \cap R(x_1,v)$, where $R(x_1,x_2)=\{x_1,x_2\}$, $x_2\neq x_1$, $R(u,x_2)\neq \{u,x_2\}$, and $R(x_2,v)\subset R(x_1,v)$.  If $x_2=v$, then we stop. Otherwise, we continue and get $x_3 \in R(u,v) \cap R(x_2,v)$, where $R(x_2,x_3)=\{x_2,x_3\}$, $x_3\neq x_2$, $R(u,x_3)\neq \{u,x_3\}$ and $R(x_3,v)\subset R(x_2,v)$.
 By repeating this step we obtain a sequence of vertices $x_0,x_1,\dots,x_q$, $q\geq 2$, such that 
\begin{enumerate}
    \item $R(x_i,x_{i+1})=\{x_i,x_{i+1}\}, i\in\{0,1,\dots,q-1\}$,
    \item $R(u,x_i) \neq \{u,x_i\}, i\in [q]$,
    \item $R(x_{i+1},v) \subset R(x_i,v), i\in\{0,1,\dots,q-1\}$.
\end{enumerate}  
Clearly, this sequence should stop by the last condition, because $V$ is finite. Hence, we may assume that $x_q=v$. Now, if $R(u,x)=\{u,x\}$, then we have a toll $u,v$-walk $uxx_1\dots x_{q-1}v$ and $x\in T(u,v)$. Otherwise, $R(u,x)\neq\{u,x\}$ and we can symmetrically build a sequence $u_0,u_1,\dots,u_r$, where $u_0=x$, $u_r=u$ and $u_0u_1\dots u_r$ is an $x,u$-path in $G_R$ that avoids $N[v]$. Clearly, $uu_{r-1}u_{r-2}\dots $ $ u_1xx_1\dots x_{q-1}v$ is a toll $u,v$-walk and $x\in T(u,v)$.

Suppose now that $x \in T(u,v)$ and $x\notin\{u,v\}$. We have to show that $x\in R(u,v)$. By Lemma~\ref{toll1} $N[u]-x$ does not separate $x$ and $v$ and $N[v]-x$ does not separate $u$ and $x$. Let $W$ be a toll $u,v$-walk containing $x$. Clearly $W$ contains an induced $u,v$-path, say $Q$. By Lemma \ref{easy} we have $V(Q)\subseteq R(u,v)$. If $x$ belongs to $Q$, then $x\in R(u,v)$. Therefore, we may assume that $x$ does not belong to $Q$. Moreover, we may assume that $x$ does not belong to any induced $u,v$-path. The underlying graph $G_R$ is $AT$-free by Theorem~\ref{at-free}. Thus, $Q$ contains a neighbor of $x$, say $x'$. If $R(u,x')=\{u,x'\}$ and $R(x',v)= \{x',v\}$, then we have a contradiction with $W$ being a toll $u,v$-walk containing $x$. Without loss of generality, we may assume that $R(x',v)\neq \{x',v\}$. If also $R(u,x')\neq \{u,x'\}$, then $x\in R(u,v)$ by the second claim of Lemma \ref{easy}. So, let now $R(u,x')=\{u,x'\}$. Since $x$ and $v$ are not separated by $N[u]-\{x\}$ by the Lemma \ref{toll1}, there exists an induced path $x,v$ $S$ without a neighbor of $u$. Let $S=s_0s_1\cdots s_k$, $s_0=x$ and $s_k=v$ and let $s_j$ be the first vertex of $S$ that also belongs to $Q$. Notice that $s_j$ can be equal to $v$ but it is different from $x'$ and that $j>0$. If $j=1$, then $x\in R(u,v)$ by Axiom (TW1) (which follows from Axiom (TW1')). If $j=2$, then $x\in R(u,v)$ by Axiom (TW1'). Hence, $j>2$. We may choose $S$ such that it minimally differs from $Q$. This means that $s_0,\dots,s_{j-2}$ may be adjacent only to $x'$ on $Q$ before $s_j$.

Suppose now that $s_j$ is adjacent to $x'$. This means that $s_j\neq v$ because $v$ is not adjacent to $x'$. Let $s_i$ be the last vertex of $S$ adjacent to $x'$ ($s_0=x$ is adjacent to $x'$). If $s_i=s_{j-1}$, then $s_{j-1}\in R(u,v)$ by Axiom (TW1). Clearly, $R(s_{j-1},u)\neq\{s_{j-1},u\}$ and $R(s_{j-1},v)\neq \{s_{j-1},v\}$ and we can use Axiom (TW2) (which follows from Axiom TW1') to get $s_{j-2}\in R(u,v)$. If we continue with the same step $j-2$ times, then we get $s_{\ell}\in R(u,v)$, respectively, for $\ell\in\{j-3,j-4,\dots,0\}$. So, $s_0=x\in R(u,v)$ and we may assume that $s_j$ is not adjacent to $x'$. Now, $s_j$ can be equal to $v$. Assume first that $s_j\neq v$. Cycle $x'x\xrightarrow{S}s_j\xrightarrow{Q}x'$ has at least six vertices and must contain some chords, since $G$ is AT-free by Theorem \ref{at-free}. If $x's_{j-1}\in E(G)$, then we get $s_0=x\in R(u,v)$ by the same steps as before (when $s_j$ was adjacent to $x'$). If $x's_{j-1}\notin E(G)$, then $x's_{j-2}\in E(G)$ and $d(x',s_j)=2$, otherwise we have an induced cycle of length at least six, which is not possible in AT-free graphs. Now, $s_{j-2}\in R(u,v)$ by Axiom (TW1'). Next, we continue $j-2$ times with Axiom (TW2) to get $s_{\ell}\in R(u,v)$, respectively, for $\ell\in\{j-3,j-4,\dots,0\}$. Again $s_0=x\in R(u,v)$ and we may assume that $s_j$ equals $v$. Again $S_{j-1}x'\in E(G)$ or $S_{j-12}x'\in E(G)$ because otherwise we have an induced cycle of length at least six, which is not possible. If $S_{j-1}x'\in E(G)$, then there exists an induced path $u,v$ in $G$ that contains $s_{j-1}$ and $s_{j-1}\in R(u,v)$ by Lemma \ref{easy}. We continue as at the beginning of this paragarph, only that we replace $s_j$ with $s_{j-1}$ (and all the other natural changes) and we get $x\in R(u,v)$. Finally, if $S_{j-1}x'\notin E(G)$, then $s_{j-2}x'\in E(G)$ again by Lemma \ref{easy}. We continue $j-2$ times with Axiom (TW2) to get $s_{\ell}\in R(u,v)$, respectively, for $\ell\in\{j-3,j-4,\dots,0\}$. Again, $s_0=x\in R(u,v)$. This completes the proof because $s_0=x\in R(u,v)$.\qed\bigskip
\end{proof}

It is easy to see that for any graph $G$, the toll walk transit function satisfies the Axioms (J2) and (TW1'). By Corollary~\ref{ATDCH2-free}, Theorems~\ref{at-free} and \ref{AT1-free} and Proposition \ref{ba} we have the following characterization of the toll walk transit function of AT-free graph.

\begin{theorem}
A transit function $R$ on a finite set $V$ satisfies the Axioms (b1'), (b2'), (J2), (J4), (J4'),  (TW1') and (TWA) if and only if $G_R$ is an AT-free graph and $R=T$ on $G_R$.
\end{theorem}

A four-cycle $axyva$ together with an edge $ua$ form a $P$-graph and a five-cycle $axybva$ together with an edge $ua$ form a $5$-pan graph. It is straightforward to check that the toll walk transit function $T$ does not satisfy Axiom (J3) on $P$-graph and $5$-pan graph. From the definitions of Axioms (J3), (J4) and (J4'), it is clear that Axiom (J3) implies both Axioms (J4) and (J4'). Therefore, we have the following corollary. 

\begin{corollary}
A transit function $R$ on a finite set $V$ satisfies Axioms (b1'), (b2'), (J2), (J3), (ba), (TW1') and (TWA) if and only if $G_R$ is an ($P$, $5$-pan,AT)-free graph and $R=T$ on $G_R$.
\end{corollary}

%%%%%%%%%%%%%%%%%%%%%%%%%%%%%%%%%%%%%%%%%%%%%%%%

\section{Toll walk transit function of Ptolemaic and distance-hereditary graphs}

Kay and Chartrand \cite{kay} introduced Ptolemaic graphs as graphs in which the distances obey the Ptolemy inequality. That is, for every four vertices $u, v, w$ and $x$ the inequality $d(u,v)d(w,x) + d(u,x)d(v,w) \geq d(u,w)d(v,x)$ holds. It was proved by Howorka \cite{edwa} that a graph is Ptolemaic if and only if it is both chordal and distance-hereditary (a graph $G$ is distance hereditary, if every induced path in $G$ is isometric). Therefore, Ptolemaic graphs are also defined as chordal graphs that are $3$ fan-free in the language of forbidden subgraphs. Consider the following axiom for the characterization of the toll walk transit function of Ptolemaic graphs. \\

\noindent\textbf{Axiom (pt).} If there exist elements $u,x,y,v \in V$ such that $x,z\in R(u,y)$, $y,z\in R(x,v)$ and $R(x,y)= \{x,y\}$, then $R(x,z)\neq \{x,z\}$ and $R(y,z)\neq\{y,z\}$.\medskip

\begin{theorem}\label{pt}
The toll walk transit function $T$ on a graph $G$ satisfies Axioms (JC) and (pt) if and only if $G$ is a Ptolemaic graph.
\end{theorem}
\begin{proof}
By Theorem~\ref{ch} $T$ satisfies Axiom (JC) if and only if $G$ is chordal. If $G$ contains an induced $3$-fan on the path $uxyv$ and the universal vertex $z$, then $x,z\in T(u,y)$, $y,z\in T(x,v)$, $T(x,y)=\{x,y\}$, $T(x,z)=\{x,z\}$ and $T(y,z) = \{y,z\}$. Hence, Axiom (pt) does not hold. That is, if $T$ satisfies Axioms (JC) and (pt), then $G$ is the Ptolemaic graph.

 Conversely, $G$ is chordal by Theorem \ref{ch} because $T$ satisfies Axiom (JC). Suppose that $T$ does not satisfy the Axiom (pt) on $G$. There exist distinct vertices $u,x,y,z,v$ such that  $x,z\in T(u,y)$ and  $y,z\in T(x,v)$, $T(x,y)= \{x,y\}$ and ($T(x,z)=\{x,z\}$ or $T(y,z)=\{y,z\}$). Without loss of generality, we may assume that $T(x,z) = \{x,z\}$. Since $x,z\in T(u,y)$ and $y,z\in T(x,v)$ there is an induced $u,x$-path $P$ without a neighbor of $y$ other than $x$, an induced $y,v$-path $Q$ without a neighbor of $x$ other than $y$, an induced $u,z$-path $R$ without a neighbor of $y$ (except possibly $z$) and an induced $z,v$-path $S$ without a neighbor of $x$ other than $z$.
 
 Now, assume that $z$ belongs to $P$, which also means that $z$ is not adjacent to $y$. Since $T(x,y) = \{x,y\}$, $y$ does not belong to $S$. Let $a$ be the common vertex of $P$ and $S$ that is close to $z$ as possible and $b$ be the common vertex of $Q$ and $S$ that is close to $y$ as possible. Note that $a$ may be the same as $z$, but $b$ is distinct from $y$. On a cycle $C:a\xrightarrow{P}zxy\xrightarrow{Q}b\xrightarrow{S}a$ is $y$ eventually adjacent only to vertices from $S$ between $a$ and $b$ (and not to $a$). In addition to that, the vertices of $S$ are not adjacent to $x$ nor to $z$. Hence, $y,x,z$ and the other neighbor of $z$ on $C$ are contained in an induced cycle of length at least four, a contradiction because $G$ is chordal.
 
 So, $z$ is not on $P$. We denote by $x'$ a neighbor of $x$ on $P$, by $x''$ the other neighbor of $x'$ on $P$ (if it exists) and by $z'$ the neighbor of $z$ on $R$. Let $a$ be the vertex common to $R$ and $P$ closest to $x$. Notice that $x'$ or $z'$ may be equal to $a$, but $z\neq a\neq x$ and $b\neq y$. If $zx'\notin E(G)$, then $zxx'x''$ is part of an induced cycle of length at least four or $x'=a$. As $G$ is Ptolemaic, $G$ is also chordal and there are no induced cycles of length four or more in $G$. So, $x'=a$. Now, if $za\notin E(G)$, then $xz'$ must be an edge to avoid an induced cycle that contains $axzz'$. In all cases, we obtain a triangle with edge $xz$: $zxx'z$ or $zxaz$ or $zxz'z$. We denote this triangle by $zxwz$.

 In addition, let $z''$ be the neighbor of $z$ on $S$ and $y'$ the neighbor of $y$ on $Q$. If $zy\notin E(G)$, then $z,x,y,y'$ and maybe some other vertices of $Q$ or $S$ induce a cycle of length at least four, which is not possible. So, $zy\in E(G)$. If $zy'\in E(G)$, then the vertices $w,x,y,y'$ and $z$ induce a $3$ fan, which is not possible in Ptolemaic graphs. Thus, $zy'\notin E(G)$. Similar, if $z''y\in E(G)$, then $w,x,y,z''$ and $z$ induce a $3$ fan. So, $z''y\notin E(G)$. Finally, the vertices $z'',z,y,y'$ possibly together with some other vertices from $S$ or $Q$ form an induced cycle of length at least four, a final contradiction. \hfill\qed\bigskip
\end{proof}

From Theorems~\ref{chordal} and \ref{pt} we have the following characterization of the toll walk function of the Ptolemaic graph.

\begin{theorem}
A transit function $R$ on a finite set $V$ satisfies Axioms (b2), (J2), (JC), (pt), (TW1), (TW2) and (TWC) if and only if $G_R$ is a Ptolemaic graph and $R=T$ on $G_R$.
\end{theorem}

We continue with the following axioms that are characteristic of the toll walk transit function $T$ on the distance-hereditary graphs.\medskip

\noindent\textbf{Axiom (dh).} If there exist elements $u,x,y,v,z \in V$ such that $x,y,z\in R(u,y)\cap R(x,v)$, $R(u,v)\neq \{u,v\}$, $R(x,y)= \{x,y\}$, $x\neq y$, $R(u,z)=\{u,z\}$, $R(v,z)=\{v,z\}$, then $ R(x,z) \neq \{x,z\}$ or $R(y,z) \neq \{y,z\}$.\medskip

\noindent\textbf{Axiom (dh1).} If there exist elements $u,x,y,v \in V$ such that    $x\in R(u,y)$  $y \in R(x,v)$,  $R(x,y)= \{x,y\}$, $x\neq y$, $R(u,x)\neq \{u,x\}$, $R(y,v)\neq \{y,v\}$, then $x\in R(u,v)$.\medskip

\begin{theorem}\label{dh}
The toll walk transit function $T$ on a graph $G$ satisfies Axioms (dh) and (dh1) if and only if $G$ is a distance-hereditary graph.
\end{theorem}

\begin{proof}
First, we prove that $T$ satisfies Axiom (dh1) if and only if $G$ is ($H$ hole $D$)-free graph. It is clear from Figure~\ref{hcdf} that $T$ does not satisfy Axiom (dh1) on $H$, hole and $D$. Conversely, suppose that $T$ does not satisfy Axiom (dh1) on $G$. That is, $x\in T(u,y)$ $y \in T(x,v)$, $T(x,y)= \{x,y\}$, $x\neq y$, $T(u,x)\neq \{u,x\}$, $T(y,v)\neq \{y,v\}$, and $x\notin T(u,v)$. Since $x\in T(u,y)$ and $T(u,x)\neq \{u,x\}$, there is an induced $u,x$-path, say $P=x_nx_{n-1}\dots x_1x_0$, where $x_0=x$ and $u=x_n$, without a neighbor of $y$ with the exception of $x$. Similar, $y \in T(x,v)$, $T(y,v)\neq \{y,v\}$ produces an induced $y,v$-path, say $Q:y_0y_1\dots y_{n-1}y_n$, where $y_0=y$ and $y_n=v$, without a neighbor of $x$ with the exception of $y$. Also, since $x\notin T(u,v)$, without loss of generality, we may assume by Lemma \ref{toll1} that the path $Q$ contains a neighbor $u'$ of $u$. We may choose $u'$ to be the neighbor of $u$ on $P$ that is closest to $y$. Then the sequence of vertices $u\xrightarrow{P}xy\xrightarrow{Q}u'u$ forms a cycle of at least five lengths. There may be chords from the vertices of $P$ to the vertices of the $y,u'$-subpath of $Q$. But $y $ is not adjacent to any vertex in $P$ and $x$ is not adjacent to any vertex in $Q$. So, some or all vertices in the sequence $u\xrightarrow{P}xy\xrightarrow{Q}u'$ induce a house if $x_1y_1\in E(G)$ and $x_2y_1 \in E(G)$ or $x_1y_2 \in E(G)$, induce a domino if $x_1y_1\in E(G)$ and $x_2y_2\in E(G)$ otherwise induce a hole.

Now we have that $G$ is ($H$ hole $D$)-free if and only if $T$ satisfies Axiom (dh1). Therefore, we have to prove that $G$ is $3$ fan-free if and only if $T$ satisfies Axiom (dh) according to Theorem \ref{disther}. If $G$ contains a $3$-fan with vertices as shown in Figure~\ref{hcdf}, then the toll walk transit function does not satisfy Axiom (dh). Conversely, suppose that $T$ does not satisfy Axiom (dh) on  (house, hole, domino)-free graph $G$. That is  $x,y,z\in T(u,y)\cap T(x,v)$,  $T(x,y)= \{x,y\}$, $x\neq y$, $T(u,z)=\{u,z\}$, $T(z,v)=\{z,v\}$ and $T(x,z) = \{x,z\}$ and $T(y,z)=\{y,z\}$. 
Since $x\in T(u,y)$, there exists an induced $u,x$-path, say $P=x_nx_{n-1}\dots x_1x_0$, where $x_0=x$ and $u=x_n$, which avoids the neighbors of $y$ with the exception of $x$ and since $y\in T(x,v)$ there exists an induced $y,v$-path, say $Q:y_0y_1\dots y_{n-1}y_n$, where $y_0=y$ and $y_n=v$, which avoids the neighbors of $x$ with the exception of $y$. Since $T(u,z)=\{u,z\}$ and $T(z,v)=\{z,v\}$, $z$ does not belong to the paths $P$ and $Q$. Let $R:uzv$ be the $u,v$-induced path containing $z$. If $R(u,x)=\{u,x\}$ and $R(y,v)=\{y,v\}$, then the vertices $u,x,y,v,z$ induce a $3$-fan. If $R(u,x)\neq \{u,x\}$ or $R(y,v)\neq \{y,v\}$, since $G$ is ($H$ hole $D$)-free the vertex $z$ is adjacent to all vertices in the path $P$ and $Q$. Then the vertices, $x,y,y_1,y_2,z$ or $x,y,x_1,x_2,z$, respectively, induce a $3$-fan graph.\qed
\end {proof}

\begin{lemma}\label{easydh}
Let $R$ be a transit function on a non-empty finite set $V$ satisfying the Axioms (J2), (J4), (dh1) and  (TW1'). If $P_n$, $n\geq 2$, is an induced $u,v$-path in $G_R$, then $V(P_n)\subseteq R(u,v)$. Moreover, if $z$ is adjacent to an inner vertex of $P_n$ that is not adjacent to $u$ or to $v$ in $G_R$, then $z\in R(u,v)$. 
\end{lemma}

\begin{proof}
If $n=2$, then $P_2=uv$ and $R(u,v)=\{u,v\}$ by the definition of $G_R$. If $n=3$, then $P_3=uxv$ and $x\in R(u,v)$ by Axiom (J2). Let now $n=4$ and $P_4=uxyv$. By Axiom (J2) we have $x\in R(u,y)$ and $y\in R(x,v)$. Now, Axiom (J4) implies that $x,y\in R(u,v)$. If $n=5$, $P_5=uxx_2yv$ and by the previous step, $x,x_2 \in R(u,y)$ and $x_2,y \in R(x,v)$. Then $x,y\in R(u,v)$ by Axiom (J4) and $x_2\in R(u,v)$ by Axiom (TW1'). If $n=6$ and $P_6=uxx_2x_3yv$, then by case $n=5$ we have $\{x,x_2,x_3 \} \in R(u,y)$ and $\{x_2,x_3,y \}\in R(x,v)$. By Axiom (J4) $x,y\in R(u,v)$ and by Axiom (TW1'), $x_2, x_3 \in R(u,v)$. For $n=7$ and $P_7=uxx_2x_3x_4yv$ we have $x_2\in R(u,x_3)$ and $x_3 \in R(x_2, v)$ by the previous cases, $R(u,x_2) \neq \{u,x_2\}$, $R(x_2,x_3) = \{x_2,x_3\}$, $R(x_3,v) \neq \{x_3,v\}$ and $x_2, x_3 \in R(u,v)$ follow ba Axiom (dh1). By the same argument we have $x_3, x_4 \in R(u,v)$. By Axiom (J4) we have $x,y\in R(u,v)$, since $x\in R(u,y)$ and $y\in R(x,v)$.
For a longer path $P_n=uxx_2\dots x_{n-2}yv$, $n>7$, we continue by induction. By the induction hypothesis we have $\{u,x,x_2,\dots,x_{n-2},y\}\subseteq R(u,y)$ and $\{x,x_3,\dots,x_{n-2},y,v\}\subseteq R(x,v)$. In particular, $x_i\in R(u,x_{i+1})$ and $x_{i+1}\in R(x_i,v)$ for every $i\in \{2,\dots,n-2\}$. By Axiom (dh1) we get $x_i,x_{i+1}\in R(u,v)$ for every $i\in\{2,\dots,n-2\}$ and by Axiom (J4) we have $x,y \in R(u,v)$.

For the second part, let $z$ be a neighbor of $x_i$, $i\in\{2,\dots,n-2\}$ that is not adjacent to $u$ and $v$. Clearly, in this case, $n\geq 5$. By the first part of the proof, we have $x_i\in R(u,v)$ and we have $z\in R(u,v)$ by Axiom (TW2).\qed
\end{proof}

\begin{proposition}\label{propdh}
If $T$ is a toll walk transit function on a distance-hereditary graph $G$, then $T$ satisfies Axioms (b2) and (TWC) on $G$.
\end{proposition}

\begin{proof}
If $T$ does not satisfy Axiom (b2), then there exist $u,v,x,y$ such that $x\in T(u,v)$, $y\in T(u,x)$ and $y\notin T(u,v)$. Since $x\in T(u,v)$, there exists an induced $x,v$-path, say $P$, without a neighbor of $u$ (except possibly $x$) and an induced $x,u$-path, say $Q$, without a neighbor of $v$ (except possibly $x$). Similarly, since $y\in T(u,x)$, there exists an induced $u,y$ path, say $R$, without a neighbor of $x$ (except possibly $y$) and an induced $y,x$ path, say $S$, without a neighbor of $u$ (except possibly $y$). Since $y\notin T(u,v)$, a neighbor of $u$ separates $y$ from $v$ or a neighbor of $v$ separates $y$ from $u$ by Lemma \ref{toll1}. But $y\xrightarrow{S} x \xrightarrow{P} v$ is a $y,v$-path that does not contain a neighbor of $u$. So, the only possibility is that a neighbor of $v$ separates $y$ from $u$. Therefore $R$ contains a neighbor of $v$, say $v'$, which is closest to $y$. If $v_1$ lies on both $R$ and $S$, then $S$ contains at least one additional vertex between $v_1$ and $x$. The vertices,  $u\xrightarrow{R}y\xrightarrow{S}x\xrightarrow{Q}u$ contain a cycle of length at least five. There may be chords from the vertices of $Q$ to both the paths, $R$ and $S$  and also from the vertices of $R$ to the vertices of $S$. Hence, some or all vertices in this sequence will induce a hole, house, domino, or fan graphs so that $T$ satisfies Axiom (b2) on distance-hereditary graphs.  

For Axiom (TWC) let $x\in T(u,v)$. There exists an induced $x,v$-path $P$ that avoids the neighborhood of $u$ (except possibly $x$). For the neighbor $v_1$ of $x$ on $P$ it follows that $v_1 \in T(x,v), v_1 \neq x$ with $T(x,v_1) = \{x,v_1\}$ and $T(u,v_1) \neq \{u,v_1\}$. If $v_1=v$, then clearly $x\notin T(v,v_1)=\{v\}$. Similarly, if $T(v_1, v)=\{v_1, v\}$, then $x\notin T(v_1,v)$. Consider next $T(v_1, v)\neq \{v_1, v\}$. We will show that $x\notin T(v_1,v)$ for a distance hereditary graph $G$. To avoid a contradiction, assume that $x\in T(v_1,v)$. There exists an induced $x,v$-path $Q$ that avoids the neighborhood of $v_1$.  The edge $xv_1$ together with some vertices of $P$ and $Q$ will form a cycle of length at least five. Also, there may be chords from vertices in $P$ to $Q$ so that these vertices may induce a hole, house, domino or a $3$-fan, a contradiction to Theorem \ref{disther}. So $x\notin T(v_1,v)$ and Axiom (TWC) hold. \qed
\end{proof}

Using Lemma~\ref{easydh}, we can modify Theorem~\ref{chord} stated as the next theorem. For this, notice that Axiom (JC) is replaced by Axioms (J4) (when $uxyv$ is a path) and (dh1) otherwise and Axioms (TW1) and (TW2) are replaced by stronger Axioms (TW1').

\begin{theorem} \label{dh1}
If $R$ is a transit function on a non-empty finite set $V$ that satisfies Axioms  (b2),  (J2), (J4), (dh1)  (TW1') and (TWC), then $R=T$ on $G_R$. 
\end{theorem}

Hence, we obtain a characterization of toll walk transit function on distance-hereditary graphs as follows. The proof follows directly by Theorems \ref{dh} and \ref{dh1}, Propositions \ref{ba} and \ref{propdh} and since Axioms (J2) and (TW1') always hold for the toll walk transit function $T$. 

\begin{theorem}\label{dh2}
A transit function $R$ on a finite set $V$ satisfies Axioms (b2), (J2), (J4), (dh), (dh1) (TW1') and (TWC) if and only if $G_R$ is a distance-hereditary graph and $R=T$ on $G_R$.
\end{theorem}

%%%%%%%%%%%%%%%%%%%%%%%%%%%%%%%%%%%%%%%%%%%%%%

\section{Non-definability of the toll walk transit function}

Here we show that it is not possible to give a characterization of the toll walk transit function $T$ of a connected graph using a set of first-order axioms defined on $R$ as we have done in the previous sections for AT-free graphs, Ptolemaic graphs, distance hereditary graphs, chordal graphs and interval graphs in \cite{lcp}. In \cite{ne-j}, Nebesky has proved that a first order axiomatic characterization of the induced path function of an arbitrary connected graph is impossible. The idea of proof of the impossibility of such a characterization is the following.\\ 
First, we construct two non-isomorphic graphs $G_d$ and $G'_{d}$ and a first-order axiom which may not be satisfied by the toll walk transit function $T$ of an arbitrary connected graph. The following axiom is defined for an arbitrary transit function $R$ on a non-empty finite set $V$ and is called the \emph{scant property} following Nebesky \cite{ne-j}.\\

\noindent\textbf{Axiom (SP).} If $R(x,y)\neq \{x,y\}$, then $R(x,y)= V$ for any $x,y\in V$.\medskip

\noindent In our case the toll walk transit function $T$ will satisfy this first order axiom on $G_d$ but not on $G'_{d}$. Then we prove by the famous $EF$ game technique of first-order nondefinability that there exists a partial isomorphism between $G_d$ and $G_{d'}$. First, we define certain concepts and terminology of first-order logic \cite{Lib}. % For more details on first order logic and game concepts refer

The tuple $\textbf{X}=(X,\mathcal{S})$ is called a \textit{structure} when $X$ is a nonempty set called \textit{universe} and $\mathcal{S}$ is a finite set of function symbols, relation symbols, and constant symbols called \textit{signature}. Here, we assume that the signature contains only the relation symbol. The \textit{quantifier rank} of a formula $\phi$ is its depth of quantifier nesting and is denoted by $qr(\phi).$ Let $\textbf{A}$ and $\textbf{B}$ be two structures with same signatures. A map $q$ is said to be a \textit{partial isomorphism} from $\textbf{A}$ to $\textbf{B}$ if and only if $dom(q)\subset A$, $rg(q)\subset B$, $q$ is injective and for any $n$-ary relation $R$ in the signature and $a_0$, $\dots$,  $a_{l-1}$ $\in dom(q)$,
 $R^\mathcal{A}( a_0, \dots,  a_{l-1} )\,\,  \text{if and only if}\,\,  R^\mathcal{B}( q{(a_0)}, \dots,  q{(a_{l-1})} ).$ 

%Let $r$ be a positive integer, and let $\mathcal{A}$ and $\mathcal{B}$ be two structures with same signature.

Let $r$ be a positive integer. The \textit{$r$-move Ehrenfeucht-Fraisse Game} on $\textbf{A}$ and $\textbf{B}$ is played between 2 players called the \textit{Spoiler} and the \textit{Duplicator}, according to the following rules.

Each run of the game has $r$ moves. In each move, Spoiler plays first and picks an element from the universe $A$ of the structure $\textbf{A}$ or from the universe $B$ of the structure $\textbf{B}$; Duplicator then responds by picking an element from the universe of the other structure. Let $a_i \in A$ and $b_i \in B$ be the two elements picked by the Spoiler and Duplicator in their $i$th move, $1\leq i \leq r$. Duplicator wins the run $(a_1,b_1),\dots,(a_r,b_r)$ if the mapping $a_i \to b_i$, where $1 \leq i \leq r$ is a partial isomorphism from the structure $\textbf{A}$ to $\textbf{B}$. Otherwise, Spoiler wins the run $(a_1,b_1),\dots,(a_r,b_r)$.

\textit{Duplicator wins the $r$-move EF-game on $\textbf{A}$ and $\textbf{B}$} or \textit{Duplicator has a winning strategy for the EF-game on $\textbf{A}$ and $\textbf{B}$} if Duplicator can win every run of the game, no matter how Spoiler plays.  

 The following theorems are our main tool in proving the inexpressibility results.

\begin{theorem}\cite{Lib}\label{ef}
The following statements are equivalent for two structures $\textbf{A}$ and $\textbf{B}$ in a relational vocabulary. 
\begin{enumerate}
\item  $\textbf{A}$ and $\textbf{B}$ satisfy the same sentence $\sigma$ with $qr(\sigma) \leq n$.
\item The Duplicator has an $n$-round winning strategy in the EF game on $\textbf{A}$ and $\textbf{B}$.
\end{enumerate}
\end{theorem}

\begin{theorem}\cite{Lib}\label{ef2}
A property \textsc{P} is expressible in first order logic if and only if there exists a number $k$ such that for every two structures $\textbf{X}$  and $\textbf{Y}$, if $\textbf{X} \in \textsc{P}$ and Duplicator has a $k$-round winning strategy on $\textbf{X}$ and  $\textbf{Y} $ then $\textbf{Y} \in \textsc{P}$.
\end{theorem}

By a \textit{ternary structure}, we mean an ordered pair $(X, D)$ where $X$ is a finite nonempty set and $D$ is a ternary relation on $X$. So $D$ is a set of triples $(x,y,z)$ for some $x,y,z\in X$. We simply write $D(x,y,z)$ when $(x,y,z)\in D$. Let $F:X\times X\to 2^X$ be defined as $F (x,y) = \{u\in X: D (x, u, y)\}$. So, for any ternary structure $(X,D)$, we can associate the function $F$ corresponding to $D$ and vice versa. If a ternary relation $D$ on $X$ satisfies the following three conditions for all $u,v,x\in X$ 
 \begin{itemize}
	 \item[$(i)$]  $D(u,u,v)$;
	 \item[$(ii)$]  $D(u,x,v) \implies D(v,x,u)$;
	 \item[$(iii)$]  $D(u,x,u) \implies x=u$,
 \end{itemize}
then the function $F$ corresponding to $D$ will be the transit function.
Observe that every axiom used in Sections 2-5 have a respective representation in terms of a ternary relation.

By the \textit{underlying graph} of a ternary structure $(X,D)$ we mean the graph $G$ with the properties that $X$ is its vertex set and distinct vertices $u$ and $v$ of $G$ are adjacent if and only if 
$$ \{x \in X:D(u,x,v)\} \cup \{x \in X:D(v,x,u)\} =\{u,v\}.$$
We call a ternary structure $(X,D)$, `the \textit{$W$ structure} of a graph $G$, if $X$ is the vertex set of $G$ and $D$ is the ternary relation corresponding to the toll walk transit function $T$ (that is $(x,y,z) \in D$  if and only if $y$ lies in some $x,z$-toll walk).  Obviously,  if $(X, D)$ is a $W$-structure, then it is the $W$-structure of the underlying graph of $(X, D)$.  We say that $(X,D)$ is \textit{scant} if the function $F$ corresponding to the ternary relation $D$, satisfies Axiom (SP) and $F$ is a transit function. 

We present two graphs $G_d$ and $G'_d$ such that the $W$-structure of one of them is scant and the other is not.  Moreover, the proof will settle, once we prove that Duplicator wins the EF game on $G_d$ and $G'_d$.
			
For $d \geq 2$ let $G_d$ be a graph with vertices and edges (indices are via modulo $4d$) as follows:
$$V(G_d)=\{u_1,u_2,\ldots,u_{4d},v_1,v_2,\ldots,v_{4d},x\} \text{ and}$$ 
$$E(G_d)=\{u_iu_{i+1},v_iv_{i+1}, u_iv_i,v_{1}x,v_{2d+1}x:i\in[4d]\}.$$ 

\begin{figure}[ht!]
\begin{center}
\begin{tikzpicture}[scale=0.8,style=thick,x=1cm,y=1cm]
\def\vr{3pt} % \vr = vertex radius;

% define vertices
%%%%% %%%%% G_d
\path (0,0) coordinate (a1);
\path (2,0) coordinate (a2);
\path (3.5,1.5) coordinate (a3);
\path (3.5,3.5) coordinate (a4);
\path (2,5) coordinate (a5);
\path (0,5) coordinate (a6);
\path (-1.5,3.5) coordinate (a7);
\path (-1.5,1.5) coordinate (a8);
\path (0.5,0.5) coordinate (b1);
\path (1.5,0.5) coordinate (b2);
\path (3,2) coordinate (b3);
\path (3,3) coordinate (b4);
\path (1.5,4.5) coordinate (b5);
\path (0.5,4.5) coordinate (b6);
\path (-1,3) coordinate (b7);
\path (-1,2) coordinate (b8);
\path (1,2.5) coordinate (x);
%  edges
\draw (a7) -- (a8) -- (a1) -- (a2) -- (a3) -- (b3) -- (b2) -- (a2);
\draw (b2) -- (b1) -- (a1);
\draw (b1) -- (b8) -- (a8);
\draw (b8) -- (b7) -- (a7);
\draw (b4) -- (b5) -- (b6) -- (a6) -- (a5) -- (a4) -- (b4);
\draw (b5) -- (a5);
\draw (b1) -- (x) -- (b5);
\draw [dashed] (a3) -- (a4);
\draw [dashed] (a6) -- (a7);
\draw [dashed] (b3) -- (b4);
\draw [dashed] (b6) -- (b7);

\draw (a1) [fill=white] circle (\vr);
\draw (a2) [fill=white] circle (\vr);
\draw (a3) [fill=white] circle (\vr);
\draw (a4) [fill=white] circle (\vr);
\draw (a5) [fill=white] circle (\vr);
\draw (a6) [fill=white] circle (\vr);
\draw (a7) [fill=white] circle (\vr);
\draw (a8) [fill=white] circle (\vr);
\draw (b1) [fill=white] circle (\vr);
\draw (b2) [fill=white] circle (\vr);
\draw (b3) [fill=white] circle (\vr);
\draw (b4) [fill=white] circle (\vr);
\draw (b5) [fill=white] circle (\vr);
\draw (b6) [fill=white] circle (\vr);
\draw (b7) [fill=white] circle (\vr);
\draw (b8) [fill=white] circle (\vr);
\draw (x) [fill=white] circle (\vr);

\draw[anchor = north] (a1) node {$u_1$};
\draw[anchor = north] (a2) node {$u_2$};
\draw[anchor = west] (a3) node {$u_3$};
\draw[anchor = west] (a4) node {$u_{2d}$};
\draw[anchor = south] (a5) node {$u_{2d+1}$};
\draw[anchor = south] (a6) node {$u_{2d+2}$};
\draw[anchor = east] (a7) node {$u_{4d-1}$};
\draw[anchor = east] (a8) node {$u_{4d}$};
\draw[anchor = south] (b1) node {$v_1$};
\draw[anchor = south] (b2) node {$v_2$};
\draw[anchor = east] (b3) node {$v_3$};
\draw[anchor = east] (b4) node {$v_{2d}$};
\draw[anchor = north] (b5) node {$v_{2d+1}$};
\draw[anchor = north] (b6) node {$v_{2d+2}$};
\draw[anchor = west] (b7) node {$v_{4d-1}$};
\draw[anchor = west] (b8) node {$v_{4d}$};
\draw[anchor = west] (x) node {$x$};

\end{tikzpicture}
\end{center}
\caption{Graph $G_d$, $d\geq 2$.} \label{G2}
\end{figure}

For $d \geq 2$ let $G'_d$ be a graph with vertices and edges as follows:
$$V(G_d)=\{u'_1,u'_2,\ldots,u'_{4d},v'_1,v'_2,\ldots,v'_{4d},x'\} \text{ and}$$
$$E(G'_d)=\{u'_1u'_{2d},u'_iu'_{i+1},u'_{2d+1}u'_{4d},u'_{2d+i}u'_{2d+i+1},v'_1v'_{2d},v'_iv'_{i+1}, v'_{2d+1}v'_{4d},$$ $$v'_{2d+i}v'_{2d+i+1},u'_jv'_j,v'_{1}x',v'_{2d+1}x':i\in[2d-1],j\in[4d] \}.$$ 

\begin{figure}[ht!]
\begin{center}
\begin{tikzpicture}[scale=0.8,style=thick,x=1cm,y=1cm]
\def\vr{3pt} % \vr = vertex radius;

%%%%%%%%%%%%%%%%%%%%G'_d
\path (8,0) coordinate (a1);
\path (5,0) coordinate (a2);
\path (5,3) coordinate (a3);
\path (8,3) coordinate (a4);
\path (11,0) coordinate (a5);
\path (14,0) coordinate (a6);
\path (14,3) coordinate (a7);
\path (11,3) coordinate (a8);
\path (7,1) coordinate (b1);
\path (6,1) coordinate (b2);
\path (6,2) coordinate (b3);
\path (7,2) coordinate (b4);
\path (12,1) coordinate (b5);
\path (13,1) coordinate (b6);
\path (13,2) coordinate (b7);
\path (12,2) coordinate (b8);
\path (9.5,1) coordinate (x);
%  edges
\draw (a1) -- (a2) -- (b2) -- (b1) -- (a1);
\draw (b2) -- (b3) -- (a3) -- (a2);
\draw (a4) -- (b4);
\draw (b5) -- (b6) -- (a6) -- (a5);
\draw (b8) -- (b5) -- (a5) -- (a8) -- (b8);
\draw (b7) -- (a7);
\draw (b1) -- (x) -- (b5);
\draw [dashed] (a3) -- (a4) -- (a1);
\draw [dashed] (a8) -- (a7) -- (a6);
\draw [dashed] (b3) -- (b4) -- (b1);
\draw [dashed] (b8) -- (b7) -- (b6);

\draw (a1) [fill=white] circle (\vr);
\draw (a2) [fill=white] circle (\vr);
\draw (a3) [fill=white] circle (\vr);
\draw (a4) [fill=white] circle (\vr);
\draw (a5) [fill=white] circle (\vr);
\draw (a6) [fill=white] circle (\vr);
\draw (a7) [fill=white] circle (\vr);
\draw (a8) [fill=white] circle (\vr);
\draw (b1) [fill=white] circle (\vr);
\draw (b2) [fill=white] circle (\vr);
\draw (b3) [fill=white] circle (\vr);
\draw (b4) [fill=white] circle (\vr);
\draw (b5) [fill=white] circle (\vr);
\draw (b6) [fill=white] circle (\vr);
\draw (b7) [fill=white] circle (\vr);
\draw (b8) [fill=white] circle (\vr);
\draw (x) [fill=white] circle (\vr);

\draw[anchor = north] (a1) node {$u'_1$};
\draw[anchor = north] (a2) node {$u'_2$};
\draw[anchor = south] (a3) node {$u'_3$};
\draw[anchor = south] (a4) node {$u'_{2d}$};
\draw[anchor = north] (a5) node {$u'_{2d+1}$};
\draw[anchor = north] (a6) node {$u'_{2d+2}$};
\draw[anchor = south] (a7) node {$u'_{4d-1}$};
\draw[anchor = south] (a8) node {$u'_{4d}$};
\draw[anchor = north] (b1) node {$v'_1$};
\draw[anchor = north] (b2) node {$v'_2$};
\draw[anchor = south] (b3) node {$v'_3$};
\draw[anchor = south] (b4) node {$v'_{2d}$};
\draw[anchor = north] (b5) node {$v'_{2d+1}$};
\draw[anchor = north] (b6) node {$v'_{2d+2}$};
\draw[anchor = south] (b7) node {$v'_{4d-1}$};
\draw[anchor = south] (b8) node {$v'_{4d}$};
\draw[anchor = north] (x) node {$x'$};

\end{tikzpicture}
\end{center}
\caption{Graph $G'_d$, $d\geq 2$.} \label{G2dash}
\end{figure}

Graphs $G_d$ and $G'_d$ are shown in Figures \ref{G2} and \ref{G2dash}, respectively.
%We generalize this observation for $G_d$ and $G_d^{\prime}$ with $d\geq 2$, in the following lemma. 
%It is easy to observe that $W$-structure of $G_d^{\prime}$ is not scant, since $T(v_2^\prime, x_1^\prime)=\{v_2^\prime, v_1^\prime, x_1^\prime \}$. 

\begin{lemma} \label{lemma}
The $W$-structure of $G_d$ is a scant and the $W$-structure of $G'_d$ is not a scant for every $d \geq 2$.
\end{lemma}

 \begin{proof}
 It is easy to observe that $W$-structure of $G'_d$ is not a scant, since $T(v'_2, x')=\{v'_2, v'_1, x'\}$. For $G_d$ let $z,y \in V(G_d)=U\cup V\cup X$, where $U=\{u_1,u_2,\ldots ,u_{4d}\}$, $V=\{ v_1,v_2,\ldots, v_{4d}\}$, $X=\{x\}$ and $d(z,y)\geq 2$. We have to show that $T(z,y) =V(G_d)$. 
    
\noindent\textbf{Case 1.} $z, y\in U$.
Let $z=u_i$ and $y=u_j$. Both $z,y$-paths on $U$ are toll walks and $U \in T(z,y)$. If we start with edge $u_iv_i$, continue on both $v_i,v_j$-paths on $V$ and end with $u_jv_j$ we get two toll $z,y$-walks that contain $V$. For $x$ notice that at least one of $z,u_1$-path or $z,u_{2d+1}$-path on $U$ contains no neighbor of $y$. We may assume that $z,u_1$-path $P$ in $U$ is such. Denote by $Q$ the $v_{2d+1},v_j$-path on $V$. Now, $zx\xrightarrow{P}u_1v_1xv_{2d+1}\xrightarrow{Q}v_jy$ is a toll walk and $T(z,y) =V(G_d)$.

\noindent\textbf{Case 2.} $z,y\in V$.
Let $z=v_i$ and $y=v_j$. By the symmetric reason as in Case 1 we have $U, V \in T(z,y)$. Again we may assume by symmetry that $z,v_1$-path $P$ on $V$ contains no neighbor of $y$. If $z\notin \{v_{2d}, v_{2d+2}\}$, then there always exists a $v_{2d+1},y$-path $Q$ on $V$ without a neighbor of $z$. Path $z\xrightarrow{P}v_1xv_{2d+1}\xrightarrow{Q}y$ is a toll walk. Otherwise, if $z\in\{v_{2d},v_{2d+2}\}$, say $z=v{2d}$, then $zv_{2d+1}xv_1\xrightarrow{Q}y$ is a toll walk if $y\neq v_{2d+2}$. So, let $z=v{2d}$ and $y=v_{2d+2}$. Now, $z\xrightarrow{P}v_1xv_1\xrightarrow{Q}y$ is a toll $z,y$-walk and we have $T(z,y)=V(G_d)$.

\noindent\textbf{Case 3.} $z=x$ and $y\in V$. Let $y=v_j$ where $j\notin\{1,2d+1\}$. Without loss of generality, let $2\leq j \leq 2d$. Now consider the following $x,v_j$-walks:
\begin{itemize}
    \item $xv_1v_2\cdots v_j$,
    \item $xv_{2d+1}v_{2d}\cdots v_j$,
    \item $xv_1u_1u_2\cdots u_jv_j$ or \\
            $xv_{2d+1}u_{2d+1}u_{2d}\cdots u_jv_j$,
    \item $xv_1u_1u_{4d}u_{4d-1} \cdots u_jv_j$ or \\
           $xv_{2d+1},u_{2d+1},u_{2d+2},\cdots ,u_{4d},u_1, u_2 \cdots u_j,v_j$,
    \item $x_1,v_1,v_{4d},v_{4d-1},\cdots , v_{2d+2}, u_{2d+2}, u_{2d+1}, u_{2d}, \cdots , u_j,v_j$ or \\
           $x_1,v_{2d+1},v_{2d+2},\cdots ,v_{4d},u_{4d},u_1, u_2 \cdots u_j,v_j$
\end{itemize}
Notice, that in the last three items only one of the mentioned walks is a toll walk when $y\in\{v_2,v_{2d}\}$. However, every vertex in $V(G_d)$ belongs to at least one toll $z,y$-walk, and $T(x,y) =V(G_d)$ follows.
      
\noindent\textbf{Case 4.} $z=x$ and $y\in U$. 
Since $u_jv_j$ is an edge, this case can be treated similarly as Case 3.
      
\noindent\textbf{Case 5.} $z\in U$ and $y\in V$. First, let $d(z,y)=2$ and we prove $T(u_1,v_2)=V(G_d)$. The following $u_1v_2$-toll walks contains every vertex of $G_d$ at least once: 
\begin{itemize}
\item $u_1u_2v_2$;
\item $u_1v_1v_2$;
\item $u_1u_{4d}u_{4d-1} \cdots u_3v_3v_2$;
\item $u_1u_{4d}v_{4d}v_{4d-1} \cdots v_{2d+1}xv_{2d+1}v_{2d}v_{2d-1}\cdots v_3v_2$.
\end{itemize}
Similarly, usually even easier, we obtain toll walks from $z$ to $y$, which will cover all vertices of $G_d$ for all the other choices of $z\in U$ and $y\in V$, also when $d(z,y) >2$. \qed
 \end{proof}

\begin{lemma}\label{lemma2}
Let $n\geq 1$ and $d> 2^{n+1}$. If $(X_1,D_1)$ and $(X_2,D_2)$ are scant ternary structures such that the underlying graph of $(X_1,D_1)$ is $G_d$ and the underlying graph of $(X_2,D_2)$ is $G_d^{\prime}$, then $(X_1,D_1)$ and $(X_2,D_2)$ satisfy the same sentence $\psi$ with $qr(\psi)\leq n$. 
\end{lemma}

\begin{proof}	
Let $X_1=\{u_1,u_2,\dots,u_{4d},v_1,v_2,\dots,v_{4d},x\}$ and let $X_2=\{u'_1,u'_2,\dots $, $u'_{4d},v'_1,v'_2,\dots,v'_{4d},x'\}$. Let $U=\{u_1, u_2,\dots,u_{4d}\}$, $V=\{ v_1,v_2,\dots,v_{4d}\}$ and $X=\{x\}$.
Also, let $U'=\{u'_1,u'_2,\dots,u'_{4d}\}$, $V'=\{v'_1,v'_2,\ldots,v'_{4d}\}$ and $X'=\{x'\}$. Clearly, $X_1=U\cup V\cup X$ and $X_2=U' \cup V'\cup X'$. Let $d^*$ and $d'$ denote the distance function of $G_d$ and $G'_d$ respectively. 
 
 We will show that the Duplicator wins the $n$-move EF-game on $G_{d}$ and $G'_{d}$ using induction on $n$. In the $i^{th}$ move of the $n$-move game on $G_{d}$ and $G'_{d}$, we use $a_i$  and $b_i$, respectively, to denote points chosen from $G_{d}$ and $G'_{d}$. Clearly, $a_i$ will be an element in $X_1$ and $b_i$ an element in $X_2$. Note that, during the game, the elements of $U$ (respectively, $V$ and $X$) will be mapped to element of $U'$ (respectively, $V'$ and $X'$).
 
 Let $H_{1}$ be the subgraph of $G_d$ induced by $U$ and $H'_{1}$ the subgraph of $G'_d$ induced by $U'$. Since $(X_1,D_1)$ and $(X_2,D_2)$ are scant ternary structures, Duplicator must preserve the edges in $G_d$ and $G'_d$ to win the game.
	
We claim that for $1\leq j,l \leq i\leq n$, Duplicator can play in $G_d$ and $G'_d$, in a way that ensures the following conditions after each round.
\begin{align*}
(1) & \text{ If } d^*(a_j,a_l) \leq 2^{n-i}, \text{ then } d'(b_j,b_l)=d^*(a_j,a_l). \\
(2) & \text{ If } d^*(a_j,a_l) > 2^{n-i}, \text{ then } d'(b_j,b_l)> 2^{n-i}.
\end{align*}	
	
Obviously, to win the game, the following correspondence must be preserved by Duplicator:
	
$$u_1 \mapsto u'_1, v_1 \mapsto v'_1, x\mapsto x', u_{2d+1}\mapsto u'_{2d+1},v_{2d+1} \mapsto v'_{2d+1}.$$ 
	
For $i=1$, (1) and (2) hold trivially. Suppose that they hold after $i$ moves and that the Spoiler makes his $(i+1)^{th}$ move. Let the Spoiler pick $a_{i+1} \in X_1$ (the case of $b_{i+1}\in X_2$ is symmetric). If $a_{i+1}=a_j$ for some $j \leq i$, then $b_{i+1}=b_j$ and conditions (1) and (2) are preserved. Otherwise, find two previously chosen vertices $a_j$ and $a_{\ell}$ closest to $a_{i+1}$ so that there are no other previously chosen vertices on the $a_j,a_{\ell}$-path of $G_d$ that passes through $a_{i+1}$. 

\noindent \textbf{Case 1.} $a_j,a_{\ell},a_{i+1}\in U$. \\
First, we consider the case where $d^*(a_j,a_{\ell}) = d_{H_1}(a_j,a_{\ell})$. (This was proved in Case 1 considered in Lemma 2 in \cite{jcap}, so we revisit the proof here.) 
If $d^*(a_j,a_{\ell})$  $ \leq 2^{n-i}$, then by the induction assumption there will be vertices $b_j$ and $b_{\ell}$ in $G'_{d}$ with $ d'(b_j,b_{\ell})\leq 2^{n-i}$. The Duplicator can select $b_{i+1}$ so that $d^*(a_j,a_{i+1})= d'(b_j,b_{i+1})$ and $d^*(a_{i+1},a_{\ell}) =d'(b_{i+1},b_{\ell})$. Clearly, the conditions $(1)$ and $(2)$ will be satisfied. On the other hand, if $d^*(a_j,a_{\ell}) > 2^{n-i},$ then by the induction assumption $d'(b_j,b_{\ell})> 2^{n-i}$. There are two cases. (i) If $d^*(a_j,a_{i+1}) > 2^{n-(i+1)} $ and $d^*(a_{i+1},a_{\ell}) > 2^{n-(i+1)}$ and fewer than $n$-rounds of the game have been played, then there exists a vertex in $G'_{d}$ at a distance larger than $2^{n-(i+1)}$ from all previously played vertices.
(ii) If $d^*(a_j,a_{i+1})\leq 2^{n-(i+1)}$ or $d^*(a_{i+1},a_{\ell})\leq 2^{n-(i+1)}$ and suppose that $d^*(a_j,a_{i+1})\leq 2^{n-(i+1)}$, then $d^*(a_{i+1},a_{\ell}) > 2^{n-(i+1)}$. So, the Duplicator can select $b_{i+1}$ with $d'(b_j,b_{i+1})= d^*(a_j,a_{i+1})$ and $d'(b_{i+1},b_{\ell})>2^{n-(i+1)}$.

Now, suppose that $d^*(a_j,a_{\ell}) \neq d_{H_1}(a_j,a_{\ell})$. This case occurs when $a_j,a_{\ell}$-shortest path contains $u_1,u_{2d+1},v_1, v_{2d+1}$ and $x$. We may assume that 
$$min \{d^*(a_j,a_{i+1}),d^*(a_{\ell}, a_{i+1})\}=d^*(a_j,a_{i+1}).$$ 
Now, choose $b_{i+1}$ so that $d^*(a_j,a_{i+1})= d'(b_j,b_{i+1})$. 

\noindent \textbf{Case 2.} $a_j,a_{\ell},a_{i+1} \in V$. \\
Let $a_j= v_r$, $a_{\ell} = v_s$, $a_{i+1} = v_t$ and find the elements $u_r$, $u_s$ and $u_t$ in $U$ and use case 1 to find the response of Duplicator when Spoiler chooses $u_t$. If $u_t \mapsto u'_z$, then choose $b_{i+1} = v'_z$. 

Similarly, for the other cases (when $a_j$ belongs to $U$ or $V$, $a_{\ell}$ belongs to $V$ or $U$ and $a_{i+1}$ belongs to $ V$ or $U$) we can make all the vertices lying in $U$ as in case 2 and it is possible to find a response from the Duplicator. Evidently, in all the cases, the conditions (1) and (2) hold. Therefore, after $n$ rounds of the game, the Duplicator can preserve the partial isomorphism.
%For, suppose $n$ rounds have been played and let $ \{p_1,$ $p_2,$ $ \ldots,$ $ p_n \} \in X_1$ and $\{q_1, \, q_2,\, \ldots, \, q_n \} \in X_2$ be the vertices chosen in the $n$-move EF-game. For, $1 \leq j,{\ell} \leq n $, let $p_jp_{\ell} \in E(G_d)$. That is, $d^*(p_j,p_{\ell})=1$, and by $(1)$ $d^{\prime}(q_j,q_{\ell})=1$. Therefore, $q_jq_{\ell} \in E(G_d^\prime)$. Conversely, let $q_jq_{\ell} \in E(G_d^\prime)$. If $p_jp_{\ell} \notin E(G_d)$, then $d^*(p_j,p_{\ell}) > 1$ and by $(2)$ we get $d^{\prime}(q_j,q_{\ell}) > 1$, contrary to our assumption that $q_jq_{\ell} \in E(G_d^\prime)$.
Thus, Duplicator wins the $n$-move EF-game on $G_{d}$ and $G'_{d}$. Hence, by Theorem \ref{ef}, we obtain the result. \qed
\end{proof}

From Lemma \ref{lemma} and Lemma \ref{lemma2}, we can conclude the following result.

\begin{theorem} \label{W}
There exists no sentence $\sigma$ of the first-order logic of vocabulary $\{D\}$ such that a connected ternary structure is a $W$-structure if and only if it satisfies $\sigma$. 
\end{theorem}

For $n\geq 1$, $d\geq 2^{n+1}$, let us consider the cycles $C_{2d}$ and $C_{2d+1}$. It is evident that the $W$-structure of both $C_{2d}$ and $C_{2d+1}$ is scant. Furthermore, Duplicator can maintain the conditions $(1)$ and $(2)$  in $C_{2d}$ and $C_{2d+1}$ and this will ensure Duplicator winning an $n$ move $EF$ game in $C_{2d}$ and $C_{2d+1}$. Since $C_{2d}$ is bipartite and $C_{2d+1}$ is not, by Theorem \ref{ef2} we arrive at the following theorem. 

\begin{theorem}
Let $(X,D)$ be a W-structure. Then the bipartite graphs cannot be defined by a first-order formula $\phi$ over $(X,D)$.   
\end{theorem}

\section{Concluding Remarks}

First, we present several examples that show the independence of the axioms used in this contribution. In all the examples we have $R(a,a)= \{a\}$ for every $a \in V$.

\begin{example} There exists a transit function that satisfies Axioms (b2'), (J2), (J4), (J4'), (TW1') and (TWA),  but not Axioms (b1') and (b1).$~$\label{ex1}\\
 Let $V=\{u,v,z,x,y\}$ and define a transit function $R$ on $V$ as follows: $R(u,v) = R(z,v)= V$,  $R(x,v)=\{x,y,v\}$, $R(u,z)=\{u,x,z\}$,   $R(u,y)=\{u,x,y\}$,  $R(z,y)=\{z,x,y\}$ and  $R(a,b)=\{a,b\}$ for all the other pairs of different elements $a,b\in V$.  It is straightforward but tedious to see that $R$ satisfies Axioms (b2'), (J2), (J4), (J4'), (TW1') and (TWA). In addition $z\in R(u,v) $, $R(u,z)\neq \{u,z\}$,  and $u\in R(z,v)$ and $R$ do not satisfy Axiom (b1'). Therefore, $R$ does not also satisfy Axiom (b1).
 \end{example}

\begin{example} There exists a transit function that satisfies Axioms (b1'), (J2), (J4), (J4'), (TW1') and (TWA),  but not Axioms (b2') and (b2).$~$\label{ex2}\\
 Let $V=\{u,v,w,x,y,z\}$ and define a transit function $R$ on $V$ as follows: $R(u,v) =\{u,y,x,v\} $, $R(u,y)=\{u,x,y\}$, $R(u,w)=\{u,y,w\}$, $R(y,v)=\{y,z,x,v\}$, $ R(u,z)= \{u,x,z\}$, $R(w,v)=\{w,z,v\}$, and $R(a,b)=\{a,b\}$ for all the other pairs of different elements $a,b\in V$. It is straightforward but tedious to see that $R$ satisfies Axioms (b1'), (J2), (J4), (J4'), (TW1') and (TWA). On the other hand $y\in R(u,v) $, $R(u,y)\neq \{u,y\}$,   $z\in R(y,v)$ and $z\notin R(u,v)$, so $R$ does not satisfy Axiom (b2') hence $R$ does not satisfy Axiom (b2).
 \end{example}

\begin{example} There exists a transit function that satisfies Axioms (b1'), (b2'), (J2),  (J4'), (TW1') and (TWA),  but not Axioms (J4) and (JC).$~$\label{ex3}\\
 Let $V=\{u,v,x,y,z\}$ and define a transit function $R$ on $V$ as follows: $R(u,v) =\{u,z,v\} $, $R(u,y)=\{u,x,y\}$,  $R(x,v)=\{x,y,v\}$ and  $R(a,b)=\{a,b\}$ for all the other pairs of different elements $a,b\in V$. It is straightforward but tedious to see that $R$ satisfies Axioms (b1'), (b2'), (J2),  (J4'), (TW1') and (TWA). In addition $x\in R(u,y) $,   $y\in R(x,v)$, $R(u,v)\neq \{u,v\}$ and $x\notin R(u,v)$, so $R$ does not satisfy Axioms (J4) and (JC).
 \end{example}

 \begin{example} There exists a transit function that satisfies Axioms (b1'), (b2'), (J2), (J4), (TW1') and (TWA),  but not Axiom (J4').$~$\label{ex4}\\
 Let $V=\{u,v,x,y,z_1,z_2,z_3\}$ and define a transit function $R$ on $V$ as follows: $R(u,v) =\{u,z_1,z_2,z_3,v\} $, $R(u,y)=\{u,x,z_1,z_2,y\}$, $R(x,v)=\{x,z_2,z_3,y,v\}$, $R(u,x)=\{u,z_1,x\}$, $R(x,y)=\{x,z_2,y\}$, $R(y,v)=\{y,z_3,v\}$, $ R(z_1,y)= \{z_1,z_2,y\}$, $R(z_3,x)= \{z_3,z_2,x\}$ and $R(a,b)=\{a,b\}$ for all the other pairs of different elements $a,b\in V$. It is straightforward but tedious to see that $R$ satisfies Axioms (b1'), (b2'), (J2), (J4), (TW1'), and (TWA). But $x\in R(u,y)$, $y\in R(x,v)$, $R(u,v)\neq \{u,v\}$, $R(u,x)\neq \{u,x\}$, $R(x,y)\neq \{x,y\}$, $R(y,v)\neq \{y,v\}$, and $x\notin R(u,v)$, so $R$ does not satisfy Axiom (J4').
 \end{example}

 \begin{example} There exists a transit function that satisfies Axioms (b1'), (b2'), (J2), (J4), (J4'),  and (TWA),  but not Axiom (TW1').$~$\label{ex5}\\
 Let $V=\{u,v,w,x,y,z\}$ and define a transit function $R$ on $V$ as follows: $R(u,v) =\{u,y,x,v\} $, $R(u,y)=\{u,x,y\}$, $R(u,w)=\{u,x,w\}$, $R(x,v)=\{x,y,v\}$, $ R(u,z)= \{u,x,z\}$,   $R(z,v)=\{z,y,v\}$,  $R(w,v)=\{w,y,v\}$ and  $R(a,b)=\{a,b\}$ for all the other pairs of different elements $a,b\in V$. It is straightforward but tedious to see that $R$ satisfies Axioms (b1'), (b2'), (J2), (J4), (J4') and (TWA). In addition,  $x,y\in R(u,v)$, $x\neq u$, $y\neq v$, $R(x,v)\neq \{x,v\}$, $R(u,y)\neq \{u,y\}$, $R(x,z)=\{x,z\}$, $R(z,w)=\{z,w\}$, $R(w,y)=\{w,y\}$ and $R(u,w)\neq \{u,w\}$, but $z\notin R(u,v)$, so $R$ does not satisfy Axiom (TW1').
 \end{example}

 \begin{example} There exists a transit function that satisfies Axioms (b1'), (b2'), (J2), (J4), (J4'), and (TW1'),  but not Axioms (TWA) and (TWC).$~$\label{ex6}\\
 Let $V=\{u,v,x,y\}$ and define a transit function $R$ on $V$ as follows: $R(u,v) =V$,  $R(x,v)=\{x,y,v\}$ and  $R(a,b)=\{a,b\}$ for all the other pairs of different elements $a,b\in V$. It is straightforward but tedious to see that $R$ satisfies Axioms (b1'), (b2'), (J2), (J4), (J4') and (TW1'). In additions $x\in R(u,v) $,  but there does not exist $x_1 \in R(x,v) \cap R(u,v)$ where $x_1 \neq x$, $R(x,x_1) = \{x,x_1\}$, $R(u,x_1) \neq \{u,x_1\}$ and $R(x_1,v) \subset R(x,v)$. Therefore, $R$ does not satisfy Axiom (TWA) and also (TWC).
 \end{example}

 \begin{example} There exists a transit function that satisfies Axioms (b1'), (b2'), (J4), (J4'), (TWA) and (TW1'),  but not Axioms (J2) and (tr).$~$\label{ex7}\\
 Let $V=\{u,v,x,y\}$ and define a transit function $R$ in $V$ as follows: $R(u,v) =\{u,x,v\} $ and $R(a,b)=\{a,b\}$ for all the other pairs of different elements $a,b\in V$. It is straightforward but tedious to see that $R$ satisfies Axioms (b1'), (b2'), (J4), (J4'), (TWA) and (TW1'). In addition $ R(u,y)=\{u,y\} $, $ R(y,v)=\{y,v\} $, $ R(u,v)\neq\{u,v\} $ but $y\notin R(u,v)$ Therefore $R$ does not satisfy Axioms (J2) and (tr).
 \end{example}

 \begin{example} There exists a transit function that satisfies Axioms (b2), (J2), (J4), (dh), (TW1), (TW2) and (TWC), but not Axioms (dh1) and (JC).$~$\label{ex8}\\
 Let $V=\{u,v,w,x,y,z\}$ and define a transit function $R$ in $V$ as follows: $R(u,v) =\{u,v\} $, $R(u,y)=\{u,z,x,y\}$, $R(u,x)=\{u,z,x\}$, $R(u,w)=\{u,z,x,y,w\}$, $R(z,y)$ $=\{z,x,y\}$, $R(z,w)=\{z,x,y,w\}$, $R(z,v)=\{z,x,y,w,v\}$, $R(x,w)=\{x,y,w\}$, $R(x,v)=\{x,y,w,v\}$, $R(y,v)=\{y,w,v\}$, and $R(a,b)=\{a,b\}$ for all other pairs of different elements $a,b\in V$. It is straightforward but tedious to see that $R$ satisfies Axioms (b2), (J2), (J4), (dh), (TW1), (TW2), and (TWC). In addition $x\in R(u,y)$, $y\in R(x,v)$, $R(u,x)\neq \{u,x\}$, $R(y,v)\neq \{y,v\}$, $R(x,y)=\{x,y\}$, and $x\notin R(u,v)$, so $R$ does not satisfy Axioma (dh1) and (JC).
 \end{example}

 \begin{example} There exists a transit function that satisfies Axioms (b2), (J2), (J4), (JC), (dh1)  (TW1), (TW2) and (TWC),  but not Axioms (dh)  and (pt).$~$\label{ex9}\\
 Let $G$ be a $3$-fan, $V=V(G)$ and define a transit function $R=T$ on $V$. It is straightforward but tedious to see that $R$ satisfies Axioms (b2), (J2), (J4), (JC), (dh1), (TW1), (TW2), and (TWC). In addition, $T$ does not satisfy the Axioms (dh) and (pt) on a $3$-fan.
 \end{example}

We conclude by observing some interesting facts about the well-known transit functions in a connected graph $G$, namely, the interval function $I$ and the induced path function $J$, and the toll walk function $T$, the topic of this paper. It easily follows that $I(u,v) \subseteq J(u,v)\subseteq T(u,v)$, for every pair of vertices $u,v$ in $G$. It is proved by Mulder and Nebesky in \cite{mune-09} that the interval function of a connected graph $G$ possesses an axiomatic characterization in terms of a set of first-order axioms framed on an arbitrary transit function. From \cite{bipartite}, it follows that an arbitrary bipartite graph also has this characterization.  Further in \cite{ch2022}, Chalopine et al. provided a first-order axiomatic characterization of $I$ of almost all central graph families in metric graph theory, such as the median graphs, Helly graphs, partial cubes, ${\ell}_1$--graphs, bridged graphs, graphs with convex balls, Gromov hyperbolic graphs, modular and weakly modular graphs, and classes of graphs that arise from combinatorics and geometry, namely basis graphs
of matroids, even $\Delta$-matroids, tope graphs of
oriented matroids, dual polar spaces. Also in \cite{ch2022}, it is proved that the family of chordal graphs, dismantlable graphs, Eulerian graphs, planar graphs, and partial Johnson graphs do not possess a first-order axiomatic characterization using the interval function $I$. The list of non-definable graph families is extended in \cite{non-definable} by including the following graphs, namely perfect, probe-chordal, wheels, odd-hole free, even-hole free, regular, $n$-colorable and $n$-connected ($n\geq 3$). It may be noted that the all-paths function $A$ also possesses an axiomatic characterization similar to that of the interval function $I$ \cite{msh}.\\
In \cite{ne-j}, Nebesky proved that the induced path function of an arbitrary connected graph does not possess such a characterization, whereas in \cite{Changat-22}, it is proved that the family of chordal graphs, Ptolemaic graphs, $(HholeP)$-free graphs, $(HholeD)$-free graphs, distance-hereditary graphs, etc. possess first-order axiomatic characterization. \\
In this paper, we have shown that the toll function $T$ does not have a first-order axiomatic characterization for an arbitrary connected graph and a bipartite graph, whereas chordal graphs, trees, $AT$-free graphs, distance hereditary graphs, and Ptolemaic graphs possess such a characterization. Graphs that possess first-order characterization also include the family of interval graphs and $(HC_5PAT)$-free graphs ~\cite{lcp}.  \\
Therefore, the behavior of these graph transit functions is strange and may not be comparable as far as axiomatic characterization is concerned. In this sense, we observe that the behavior of the induced path function may be comparable to the toll function to some extent.  Since most of the classes of graphs that we have provided axiomatic characterizations in terms of the toll function are related to $AT$-free graphs, we believe that the following problem will be relevant.  \\\\
\textbf{Problem.}
It would be interesting to check whether some of the maximal subclasses of $AT$-free graphs like $AT$-free $\cap$ claw-free, strong asteroid-free graphs and the minimal superclasses of $AT$-free graphs like the dominating pair graphs and the probe $AT$-free graphs possess a first-order axiomatic characterization in terms of the toll function $T$? \\

\noindent\textbf{Acknowledgments}:  L.K.K.S acknowledges the financial support of CSIR, Government of India, for providing CSIR Senior Research Fellowship (CSIR-SRF) ({No 09/102(0260)/2019-EMR-I} ). J.J acknowledges the financial support of the University of Kerala, India, for providing University JRF (No: 445/2020/UOK, 391/2021/UOK, 3093/2022/ UOK, 4202/2023/UOK). I.P. was partially supported by Slovenian Research and Inovation Agency by research program number P1-0297.

\end{document}